\documentclass[11pt]{amsart}  

\usepackage{mathptmx}
\usepackage[round]{natbib}  
\usepackage{latexsym,amssymb,amscd,amsmath,mathrsfs,bbm} 
\usepackage[all]{xy}

\newtheorem{theorem}{Theorem}[section]
\newtheorem{lemma}[theorem]{Lemma}
\newtheorem{proposition}[theorem]{Proposition}
\newtheorem{corollary}[theorem]{Corollary}
\theoremstyle{definition}
\newtheorem{definition}[theorem]{Definition}
\newtheorem{example}[theorem]{Example}
\theoremstyle{remark}
\newtheorem{remark}[theorem]{Remark}
  
\def\La{\Lambda}   
\def\Ga{\Gamma}       \def\De{\Delta}
     
       \def\th{\theta}
\def\al{\alpha}       
\def\be{\beta}        \def\eps{\varepsilon}
\def\ga{\gamma}       
\def\la{\lambda}      \def\ka{\kappa}
\def\vi{\varphi}      
\def\si{\sigma}       \def\om{\omega}

\def\mA{\mathbb A} \def\mN{\mathbb N}
 
 \def\mP{\mathbb P}

\def\mL{\mathbb L} 
 \def\mZ{\mathbb Z}

\def\aK{\mathbbm k}

 \def\gP{\mathfrak p}
 \def\gQ{\mathfrak q}
 \def\gR{\mathfrak r}

\def\gM{\mathfrak m}

\def\dM{\mathfrak M} 

\def\kA{\mathcal A} \def\kN{\mathcal N}
\def\kB{\mathcal B} \def\kO{\mathcal O}
\def\kC{\mathcal C} \def\kP{\mathcal P}
\def\kD{\mathcal D} \def\kQ{\mathcal Q}
\def\kE{\mathcal E} \def\kR{\mathcal R}
\def\kF{\mathcal F} \def\kS{\mathcal S}
\def\kG{\mathcal G} \def\kT{\mathcal T}
\def\kH{\mathcal H} \def\kU{\mathcal U}
\def\kI{\mathcal I} \def\kV{\mathcal V}
\def\kJ{\mathcal J} 
\def\kK{\mathcal K} 
\def\kL{\mathcal L} 
\def\kM{\mathcal M}

\def\bC{\mathbf C}

\def\bL{\mathbf L}

\def\fA{\mathbf a} \def\fN{\mathbf n}

\def\sB{\mathsf B} 
 
\def\sD{\mathsf D} 
 \def\sR{\mathsf R}
\def\sF{\mathsf F} 
\def\sG{\mathsf G} 
\def\sH{\mathsf H} 
 
\def\sJ{\mathsf J} 
 
\def\sL{\mathsf L}

\def\rA{\mathrm A}

 \def\rR{\mathrm R}

\def\rL{\mathrm L}

\def\cA{\mathscr A} 
\def\cB{\mathscr B} 
\def\cC{\mathscr C} 
\def\cD{\mathscr D}

 \def\cT{\mathscr T}
\def\cH{\mathscr H\hskip-3pt\mathit o}

\def\qA{{\mbox{\boldmath$A$}} }
\def\qB{{\boldsymbol B}} \def\qO{{\boldsymbol O}}
\def\qC{{\boldsymbol C}} 
\def\qD{{\boldsymbol D}} 
 \def\qR{{\boldsymbol R}}
 \def\qS{{\boldsymbol S}}
 
\def\qH{{\boldsymbol H}} 
\def\qI{{\boldsymbol I}} 
 
\def\qK{{\boldsymbol K}}

\def\okB{{\bar{\kB}}}	
	\def\oqA{{\boldsymbol{\bar{A}}}}
	
	\def\okA{\bar\kA}

\def\tkA{\tilde{\kA}}
\def\tsG{\tilde{\sG}}		\def\tLa{\tilde{\varLambda}}
		\def\tkP{\ti\kP}

\def\hqH{\boldsymbol{\hat{H}}}

\def\ukP{\underrightarrow{\mathcal{P}}}

\def\bde{\boldsymbol{\varDelta}}

\def\osG{\bar\sG}	\def\osF{\bar\sF}

\def\tsF{\tilde{\mathsf F}}
\def\tsH{\tilde{\mathsf H}}  \def\tsG{\tilde{\mathsf G}}
	\def\tkT{\tilde{\kT}}
	
	\def\omL{\overline{\mL}}

\def\8{\infty}		\def\0{\emptyset}
\def\mt{\mbox{-}}

\def\lb{\textup{(}} 	\def\rb{\textup{)}}

\def\bap{\bigcap}        \def\bup{\bigcup}
\def\sb{\subset}         \def\sp{\supset}
\def\spe{\supseteq}      \def\sbe{\subseteq}
\def\xx{\times}			\def\+{\oplus}
\def\*{\otimes}			\def\bop{{\textstyle\bigoplus}}
\def\cop{\textstyle{\bigsqcup}}

\def\mps{\mapsto}		\def\ti{\tilde}
\def\xarr{\xrightarrow}
\def\Arr{\Rightarrow}

\def\til{\hskip-1pt\tilde{\phantom{a}}}

\def\1{\mathbbm 1}

\def\cl{_\mathrm{cl}}
\def\LF{\mathsf{LF}}	
	\def\RF{\mathsf{RF}}
\def\RH{\mathsf{RH}}	\def\DG{\mathsf{DG}}

\def\gdim{\mathop\mathrm{gl.dim}\nolimits}

\def\idim{\mathop\mathrm{inj.dim}\nolimits}
\def\pdim{\mathop\mathrm{pr.dim}\nolimits}

\def\kdim{\mathop\mathrm{Kr.dim}\nolimits}
\def\lpdim{\mathop\mathrm{lp.dim}\nolimits}

\def\ob{\mathop\mathrm{Ob}\nolimits}
\def\add{\mathop\mathrm{add}\nolimits}
\def\Add{\mathop\mathrm{Add}\nolimits}

\def\per{\mathsf{Perf}}	
	
\def\hom{\mathop\mathrm{Hom}\nolimits}
\def\rhom{\mathop\mathrm{RHom}\nolimits}
\def\ext{\mathop\mathrm{Ext}\nolimits}
\def\End{\mathop\mathrm{End}\nolimits}
\def\con{\mathop\mathrm{Cone}\nolimits}

\def\ann{\mathop\mathrm{ann}\nolimits}
\def\ass{\mathop\mathrm{Ass}\nolimits}
\def\trs{\mathop\mathrm{tors}\nolimits}
\def\Hom{\mathop{\mathcal{H}\!\mathit{om}}\nolimits}
\def\rHom{\mathop{\mathrm{R}\mathcal{H}\!\mathit{om}}\nolimits}
\def\Ext{\mathop{\mathcal{E}\!\mathit{xt}}\nolimits}
\def\tf{\mathop\mathsf{tf}\nolimits}
\def\tors{\mathop\mathsf{tors}\nolimits}

\def\ker{\mathop\mathrm{Ker}\nolimits}
\def\im{\mathop\mathrm{Im}\nolimits}
\def\cok{\mathop\mathrm{Cok}\nolimits}
\def\rad{\mathop\mathrm{rad}\nolimits}

\def\qoh{\mathop\mathsf{Qcoh}\nolimits}
\def\lp{\mathop\mathsf{lp}\nolimits}

\def\spec{\mathop\mathrm{Spec}\nolimits}
\def\END{\mathop{\mathcal{E}\!\mathit{nd}}\nolimits}
\def\Mat{\mathop\mathrm{Mat}\nolimits}
\def\supp{\mathop\mathrm{supp}\nolimits}
\def\cen{\mathop\mathrm{cen}\nolimits}
\def\kinj{\mathop\mathsf{K\mbox{-}inj}\nolimits}

\def\cm{\mathop\mathrm{CM}\nolimits}

\def\md{\mbox{-}\mathsf{mod}}
\def\Md{\mbox{-}\mathsf{Mod}}

\def\inj{\mbox{-}\mathsf{Inj}}
\def\Inj{^\mathsf{Inj}}

\def\op{^\mathrm{op}}

\def\tl{\stackrel{\scriptstyle \sL}{\otimes}}
\def\ito{\stackrel{\sim}{\to}}
\def\cir{\mathop{\stackrel{\scriptscriptstyle\circ}{}}}
\def\ch{^{\scriptscriptstyle\vee}}			\def\sh{^\sharp}
\def\sd{^{\scriptscriptstyle\bullet}}
\def\bu{{\scriptscriptstyle\bullet}}

\def\smtr#1{\left(\begin{smallmatrix}#1\end{smallmatrix}\right)}
\def\set#1{\left\{\,#1\,\right\}}
\def\setsuch#1#2{\left\{\,#1\mid #2\,\right\}}
\def\gnr#1{\langle\,#1\,\rangle}

\def\lst#1#2{ #1_1 , #1_2 , \dots , #1_{#2} }
\def\row#1#2{( #1_1 , #1_2 , \dots , #1_{#2} )}
\def\mtr#1{\begin{pmatrix}#1\end{pmatrix}}
\def\={\setminus}
	
\def\comp{\complement}

\def\sm{^{\mathrm{c}}}

\def\dlim{\varinjlim}

\def\TF{torsion free}
\def\ncs{non-com\-mu\-ta\-tive scheme}
\def\ncc{non-com\-mu\-ta\-tive curve}
\def\ncss{non-com\-mu\-ta\-tive schemes}
\def\nccs{non-com\-mu\-ta\-tive curves}
\def\qch{quasi-coherent}
\def\cht{coherent}
\def\qis{quasi-iso\-morphism}
\def\qhr{quasi-hereditary}
\def\iff{if and only if }
\def\lop{locally projective}
\def\flop{locally projective and locally finitely generated}

\def\gor{strong\-ly Goren\-stein}
\def\meq{Morita equivalent}
\def\meqc{Morita equivalence}
\def\CM{Cohen-Macaulay}
\def\dvr{discrete valuation ring}

\begin{document}
\bibliographystyle{plainnat}

\title{Minors and resolutions of non-commutative schemes}

\author[I.\,Burban, Y.\,Drozd, V.\,Gavran]{Igor Burban    \and Yuriy Drozd
        \and Volodymyr Gavran  }
\address{Mathematisches Institut, 
Universit\"at zu K\"oln,
Weyertal 86-90,
 K\"oln, D-50931, Germany.}
\address{Institute of Mathematics,
National Academy of Sciences of Ukraine,
Tereschenkivska str. 3,\\
 Kyiv, 01601, Ukraine.}
\email{burban@math.uni-koeln.de}
\email{y.a.drozd@gmail.com,\,drozd@imath.kiev.ua}
\email{vlgvrn@gmail.com}

\begin{abstract}
In this article we develop the theory of minors of non-commutative schemes. This study is motivated by applications in the theory of non-commutative resolutions of singularities of commutative schemes. In particular, we construct a categorical resolution for \nccs\ and in the
rational case show that it can be realized as the derived category of a quasi-hereditary algebra.
\end{abstract}

\keywords{Derived categories \and bilocalization \and non-commutative schemes \and minors}
\subjclass{14F05, 14A22}

\maketitle

\section{Introduction}

Let $B$ be a ring and $P$ be a finitely generated projective left $B$-module. We call the
 ring $A = B_P = \bigl(\mathrm{End}_B P\bigr)^{\mathrm{op}}$ a \emph{minor} of $B$. It turns out that 
 the module categories of $B$ and $A$ are closely related. 

\begin{enumerate}
\item The functors $\sF = P \otimes_A\_$ and $\sH = \hom_A(P^\vee,\_)$ from $A\Md$ to
$B\Md$ are fully faithful, where $P^\vee = \hom_B(P, B)$. In other words, $A\Md$ can be realized in two different ways as a 
full subcategory of $B\Md$, see Theorem~\ref{FG}.

\item The functor $\sG = \hom_B(P,\_): B\Md \longrightarrow A\Md$ is exact and essentially surjective. Moreover, we have adjoint pairs $(\sF, \sG)$ and $(\sG, \sH)$. In other words, $\sG$ is a \emph{bilocalization functor}. If $$I = I_P = \mathrm{Im}\bigl(P \otimes_A P^\vee \longrightarrow B\bigr)$$ and $\bar{B} = B/I$ then
    the category $\bar{B}\Md$ is the kernel of $\sG$ and $A\Md$ is equivalent to the Serre
    quotient of $B\Md$ modulo $\bar{B}\Md$, see Theorem \ref{FG}.\textit{2}.
    
\item Under certain additional assumptions one can show that the global dimension of $B$ is finite provided the global dimensions
of $A$ and $\bar{B}$ are finite, see Lemma \ref{gdim}.
\end{enumerate}
The described  picture becomes even better when we pass to the (unbounded) derived categories
$\cD\bigl(A\Md\bigr)$, $\cD\bigl(B\Md\bigr)$ and $\cD\bigl(\bar{B}\Md\bigr)$
of the rings $A, B$ and $\bar{B}$ introduced above. Let $\mathsf{DG}$ be the derived functor of $\sG$, $\mathsf{LF}$ be the left derived functor of $\sF$ and
$\mathsf{RH}$ be the right derived functor of $\sH$.
\begin{enumerate}
\item  Then we have adjoint pairs $(\mathsf{LF}, \mathsf{DG})$ and $(\mathsf{DG}, \mathsf{RH})$, the functors $\mathsf{LF}$ and  $\mathsf{RH}$ are fully faithful and the category
$\cD\bigl(A\Md\bigr)$ is equivalent to the Verdier localization of $\cD\bigl(B\Md\bigr)$
modulo its triangulated subcategory $\cD_{\bar{B}}\bigl(B\Md\bigr)$ consisting of complexes with cohomologies from
$\bar{B}\Md$, see Theorem \ref{DFG}.
\item Moreover, we have a semi-orthogonal decomposition
$$
\cD\bigl(B\Md\bigr) = \bigl\langle \cD_{\bar{B}}\bigl(B\Md\bigr), \, \cD\bigl(A\Md\bigr)
\bigr\rangle,
$$
see Corollary~\ref{14}.
\end{enumerate}
One motivation to deal with minors comes from the theory of non-commu\-ta\-tive crepant resolutions. Let $A$
be a commutative normal Gorenstein domain  and $F$ be a reflexive $A$--module such that the ring
\begin{equation*}
B =  B_F: = \End_A(A \oplus F)^{\mathrm{op}} =
\left(
\begin{array}{cc}
A & F \\
F^{\vee} & E
\end{array}
\right),
\end{equation*}
where $E = \bigl(\End_A F\bigr)^{\mathrm{op}}$, is maximal Cohen--Macaulay over $A$ and of finite global dimension. Van den Bergh suggested to view
$B$ as a \emph{non-commutative crepant resolution} of $A$ showing that under some additional assumptions,
the existence of a non-commutative crepant resolution implies the existence of a commutative one \citep{vdb}. If we take
the idempotent
$
e = \left(\begin{smallmatrix} 1 & 0 \\ 0 & 0\end{smallmatrix}\right) \in B
$
and pose $P = Be$ then it is easy to see that $A = B_P$. Thus, dealing with non-commutative (crepant) resolutions of singularities, we naturally come
into the framework of the theory of minors.

In \citep{dg} it was observed that there is a close relation between coherent sheaves over the nodal cubic
$C = V(zy^2 - x^3 - x^2z) \subset \mathbb{P}^2$ and representations of  the finite dimensional 
algebra $\Lambda$ given by the quiver with relations
\begin{equation*}
\xymatrix{
\bullet   \ar@/^/[r]^{\alpha_1} \ar@/_/[r]_{\alpha_2}  & \bullet \ar@/^/[r]^{\beta_1} \ar@/_/[r]_{\beta_2} & \bullet
} \quad \beta_1 \alpha_1 =  \beta_2 \alpha_2 = 0.
\end{equation*}
An explanation of this fact was given in \citep{bd}. Let  $\kI$ be the ideal sheaf of the singular point of $C$ and $\kA = \END_{C}(\kO \oplus \kI)$. Consider the ringed space $(C, \kA)$ and the category $\kA\md$  of coherent left $\kA$--modules on $C$. 
The the derived category $\cD^b\bigl(\kA\md\bigr)$  has a tilting complex, whose (opposite) endomorphism algebra is isomorphic to $\Lambda$ what implies that the categories 
$\cD^b\bigl(\kA\md\bigr)$ and $\cD^b\bigl(\Lambda\md)$ are equivalent. On the other hand, the
triangulated category $\per(C)$ of \emph{perfect complexes} on $C$ is equivalent to a full subcategory of $\cD^b\bigl(\kA\md\bigr)$. In fact, we deal  here with a sheaf-theoretic version of the construction of minors: the commutative scheme  $(C, \kO)$ is a minor of the non-commutative
scheme $(C, \kA)$. The goal of this article is to establish a general framework for the theory of minors of non-commutative schemes.

In Section \ref{s1} we review some key results on localizations of abelian and triangulated categories used in this article.
In Section \ref{s2} we discuss the theory of non-commutative schemes, elaborating in particular a  proof of the result  characterizing the triangulated category $\per(\kA)$ of perfect complexes over a non-commutative scheme
$(X, \kA)$ as the category of \emph{compact objects} of the unbounded derived category of quasi-coherent sheaves
$\cD(\kA)$ (Theorem \ref{perf}).  Section \ref{s3} is devoted to the definition of a minor $(X, \kA)$ of a non-commutative scheme $(X, \kB)$ and the study of relations between $(X, \kA)$ and $(X, \kB)$. In Section~\ref{3a} we introduce the notion of \emph{quasi-hereditary} \ncss,
which generalizes the notions of quasi-hereditary semiprimary rings \citep{cps,dr} and quasi-hereditary orders \citep{ko1} and study
their properties. In Section \ref{s4} we elaborate the theory of \emph{strongly Gorenstein} non-commutative schemes. 
Section \ref{s5} deals with non-commutative curves. In particular, we study here hereditary non-commutative curves. In the final 
Section \ref{konig}, as an application of the elaborated technique, we construct a categorical resolution for any (reduced) \ncc\
(Theorem~\ref{k1}). We call it \emph{K\"onig resolution}, since it is an analogue of the construction proposed by K\"onig \citep{ko2}.
  If this curve is rational, we construct a tilting complex, which shows that this categorical resolution can be realized as the derived category
  of modules over a finite dimensional quasi-hereditary algebra (Theorem~\ref{tilt}). In particular, it gives an estimate of the Rouquier 
  dimension of the perfect derived category of coherent sheaves over a \ncc\ (Corollary~\ref{ti2}). For ``usual'' (commutative) curves
  this result is contained in \citep{bdg}.

\section{Bilocalizations}
\label{s1}

 Recall that a full subcategory $\cC$ of an abelian category $\cA$ is said to be \emph{thick} (or 
 \emph{Serre subcategory}) if, for any exact sequence $0\to C'\to C\to C''\to 0$, the object $C$ belongs
 to $\cC$ \iff both $C'$ and $C''$ belong to $\cC$. Then the \emph{quotient category} $\cA/\cC$ is defined
 and we denote by $\Pi_\cC$ the natural functor $\cA\to\cA/\cC$. It is exact, essentially surjective and 
 $\ker\Pi_\cC=\cC$. For instance, if $\sG:\cA\to\cB$ is an exact functor among abelian categories,
 its kernel $\ker\sG$ is a thick subcategory of $\cA$ and $\sG$ factors as $\bar\sG\cir\Pi_{\ker\sG}$,
 where $\bar\sG:\cA/\ker\sG\to\sB$.
 
 Analogously, if $\cC$ is a full subcategory of a triangulated category $\cA$, it is said to be 
 \emph{thick} if it is triangulated (i.e. closed under shifts and cones) and closed under taking direct
 summands. Then the \emph{quotient triangulated category} $\cA/\cC$ is defined and we denote by 
 $\Pi_\cC$ the natural functor $\cA\to\cA/\cC$. It is exact (triangulated), essentially surjective and 
 $\ker\Pi_\cC=\cC$. For instance, if $\sG:\cA\to\cB$ is an exact (triangulated) functor among 
 triangulated categories, its kernel $\ker\sG$ is a thick subcategory of $\cA$ and $\sG$ factors as 
 $\bar\sG\cir\Pi_{\ker\sG}$, where $\bar\sG:\cA/\ker\sG\to\sB$.
 
 If $\sF:\cA\to\cB$ is a functor, we denote by $\im\sF$ its \emph{essential image}, i.e. the full subcategory
 of $\cB$ consisting of objects $B$ such that there is an isomorphism $B\simeq\sF A$ for some $A\in\cA$.
 We usually use this term when $\sF$ is a full embedding (i.e. is fully faithful), so $\im\sF\simeq\cA$.
 
 We use the following well-known facts related to these notions.
 
   \begin{theorem}\label{11} 
 \begin{enumerate}
 \item Let $\cA,\cB$ be abelian categories, $\sG:\cA\to\cB$ be an exact functor which has a left adjoint \lb right adjoint\rb\ 
 $\sF:\cB\to\cA$ such that the natural morphism $\1_\cB\to \sG\cir \sF$ \lb respectively,
 $\sG\cir\sF\to\1_\cB$\rb\ is an isomorphism. Let $\cC=\ker\sG$.
	\begin{enumerate}
	 \item $\sG=\osG\cir\Pi_\cC$, where $\osG$ is an equivalence $\cA/\cC\ito\cB$ and its quasi-inverse 
	 functor is $\osF=\Pi_\cC\cir\sF$.
	 \item $\sF$ is a full embedding and its essential image $\im\sF$ coincides with the left 
	 \lb respectively, right\rb\ orthogonal subcategory of $\cC$, i.e. the full subcategory
	 \[ \hspace*{2.5em}
	 {^\perp\cC}=\setsuch{A\in\ob\cA}{\hom(A,C)=\ext^1(A,C)=0 \text{\emph{ for all }} C\in\ob\cC}
	 \] 
	\lb respectively, 
	\[  \hspace*{2.5em}
	\cC^\perp=\setsuch{A\in\ob\cA}{\hom(C,A)=\ext^1(C,A)=0 \text{\emph{ for all }} C\in\ob\cC}.)
	\]
	\item  $\cC=({^\perp\cC})^\perp$ \lb{}respectively, $\cC={^\perp(\cC^\perp)}$\,\rb.
	\item  The embedding functor $\cC\to\cA$ has a left \lb respectively, right\rb\ adjoint.
	\end{enumerate} 

\smallskip
 \item Let $\cA,\cB$ be triangulated categories, $\sG:\cA\to\cB$ be an exact \lb triangulated\rb\ functor 
  which has a left adjoint \lb right adjoint\rb\ $\sF:\cB\to\cA$ such that the natural morphism 
 $\1_\cB\to \sG\cir \sF$ \lb respectively, $\sG\cir\sF\to\1_\cB$\rb\ is an isomorphism. Let 
 $\cC=\ker\sG$. 
 \begin{enumerate}
	 \item $\sG=\osG\cir\Pi_\cC$, where $\osG$ is an equivalence $\cA/\cC\ito\cB$ and its quasi-inverse 
	 functor is $\osF=\Pi_\cC\cir\sF$.
	 \item $\sF$ is a full embedding and its essential image $\im\sF$ coincides with the left 
	 \lb respectively, right\rb\ orthogonal subcategory of $\cC$, i.e. the full subcategory%
 \footnote{\,Note that in the book \citep{ne} the notations for the orthogonal subcategories are opposite to ours. 
 The latter seems more usual, especially in the representation theory, see, for instance, \citep{ajs,gl2}. 
 In \citep{gab} the objects of the right orthogonal subcategory $\cC^\perp$ are called 
 $\cC$\emph{-closed}.}
	 \[
	 {^\perp\cC}=\setsuch{A\in\ob\cA}{\hom(A,C)=0 \text{\emph{ for all }} C\in\ob\cC}
	 \] 
	\lb respectively, 
	\[
	\cC^\perp=\setsuch{A\in\ob\cA}{\hom(C,A)=0 \text{\emph{ for all }} C\in\ob\cC}.)
	\]
	\item  $\cC=({^\perp\cC})^\perp$ \lb{}respectively, $\cC={^\perp(\cC^\perp)}$\,\rb.
	\item  The embedding functor $\cC\to\cA$ has a left \lb respectively, right\rb\ adjoint, which
	induces an equivalence $\cA/{^\perp\cC}\ito\cC$ \lb respectively, $\cA/{\cC^\perp}\ito\cC$\rb.
	\end{enumerate} 
 \end{enumerate}
 \end{theorem}
 \begin{proof}
 The statement (1a) is proved in \cite[Ch.~III,~Proposition~5]{gab} if $\sF$ is right adjoint of $\sG$. 
 The case of left adjoint is just a dualization. The proof of the statement (2a) is quite analogous.
 Therefore, from now on we can suppose that $\cB=\cA/\cC$. Then
 the statements (1b) and (2b) are just \cite[p.371, Lemma 2 et Corollaire]{gab} and 
 \cite[Theorem 9.1.16]{ne}. The statements (1c) and (2c) are \cite[Corollary 2.3]{gl2} and
 \cite[Corollary 9.1.14]{ne}. Thus the statement (2d) also follows from \cite[Theorem 9.1.16]{ne}.
 In the abelian case the left (respectively, right) adjoint $\sJ$ to the embedding $\cC\to\cA$ 
 is given by the rule $A\mps\cok\Psi(A)$ (respectively, $A\mps\ker\Psi(A)$\,), where $\Psi$ 
 is the natural morphism $\sF\cir\sG\to\1_\cA$ (respectively, $\1_\cA\to\sF\cir\sG$).
 \end{proof}
 
 \begin{remark}
 Note that in the abelian case the composition $\Pi_{^\perp\cC}\cir\sJ$ (respectively, 
 $\Pi_{\cC^\perp}\cir\sJ$) need not be an equivalence. The reason is that the subcategory 
 ${^\perp\cC}$ ($\cC^\perp$) need not be thick (see \citep{gl2}).  
 \end{remark}

 A thick subcategory $\cC$ of an abelian or triangulated category $\cA$ is said to be \emph{localizing} 
 (\emph{colocalizing}) if the canonical functor $\sG:\cA\to\cA/\cC$ has a right (respectively, left)
 adjoint $\sF$. Neeman \citep{ne} calls $\sF$ a \emph{Bousfield localization} (respectively, a 
 \emph{Bousfield colocalization}).\!%
 \footnote{\,Actually, Neeman uses this term for triangulated categories, but we will use it for 
 abelian categories too.}
  In this case the natural morphism $\sG\cir\sF\to\1_{\cA/\cC}$ 
 (respectively, $\1_{\cA/\cC}\to\sG\cir\sF$) is an isomorphism \cite[Ch.III,Proposition 3]{gab}, 
 \cite[Lemma 9.1.7]{ne}. If $\cC$ is both localizing and colocalizing, we call it 
 \emph{bilocalizing} and call the category $\cA/\cC$ (or any equivalent one) a 
 \emph{bilocalization} of $\cA$. We also say in this case that $\sG$ is a \emph{bilocalization functor}.
 In other words, an exact functor $\sG:\cA\to\cB$ is a bilocalization functor if it has both left adjoint
 $\sF$ and right adjoint $\sH$ and the natural morphisms $\1_\cB\to\sG\sF$ and $\sG\sH\to\1_\cB$ are 
 isomorphisms.
 
 \begin{corollary}\label{biloc} 
   Let $\sG:\cA\to\cB$ be an exact functor between abelian or triangulated categories which has both left
   adjoint $\sF$ and right adjoint $\sH$. In order that $\sG$ will be a bilocalization functor it is necessary
   and sufficient that one of the natural morphisms $\1_\cB\to\sG\cir\sF$ or $\sG\cir\sH\to\1_\cB$ be an isomorphism.
  \end{corollary} 
  \begin{proof}
   Let, for instance, the first of these morphisms be an isomorphism. Then there is an equivalence of
   categories $\osG:\cA/\ker\sG\ito\cB$ such that $\sG=\osG\Pi_\cC$, where $\cC=\ker\sG$. So we can suppose
   that $\cB=\cA/\cC$ and $\sG=\Pi_\cC$. Thus the morphism $\sG\sH\to\1_\cB$ is an isomorphism, since $
   \sH$ is right adjoint to $\sG$.
  \end{proof}
 
 \begin{corollary}\label{12} 
 Let $\cC$ be a localizing \lb colocalizing\rb\ thick subcategory of an abelian category $\cA$, 
 $\cD_\cC(\cA)$ be the full subcategory of $\cD(\cA)$ consisting of all complexes $C\sd$ such that 
 all cohomologies $H^i(C\sd)$ are in $\cC$. 
 Suppose that the Bousfield localization \lb{}respectively, colocalization\rb\  functor
 $\sF$ has right \lb{}respectively, left\rb\ derived functor.
 Then $\cD_\cC(\cA)$ is also a localizing \lb colocalizing\rb\ subcategory of $\cA$
 and $\cD(\cA/\cC)\simeq\cD(\cA)/\cD_\cC(\cA)$.
 \end{corollary} 
 \begin{proof}
  We consider the case of a localizing subcategory $\cC$, denote by $\sG$ the canonical functor $\cA\to\cA/\cC$
  and by $\sF$ its right adjoint. As $\sG$ is exact, it induces an exact functor $\cD(\cA)\to\cD(\cA/\cC)$
  acting on complexes componentwise. We denote it by $\DG$; it is both right and left derived of $\sG$.
  Obviously, $\ker\sD\sG=\cD_\cC(\cA)$. Since $\sG\cir\sF\to\1_{\cA/\cC}$ is an isomorphism,
  the morphism $\DG\cir\RF\to\1_{\cD(\cA/\cC)}$ is also an isomorphism, so we can apply 
  Theorem\ref{11}\,(2).
 \end{proof}
 
 \begin{remark}\label{13} 
 \begin{enumerate}
 \item   If $\cC$ is localizing and $\cA$ is a Grothendieck category, the right derived 
 functor $\RF$ exists \citep{ajs}, so $\cD(\cA/\cC)\simeq\cD_\cC(\cA)$.
   We do not know general conditions which ensure the existence 
 of the left derived functor $\LF$ in the case of colocalizing categories,
 though it exists when $\cA$ is a category of \qch\ modules over
 a quasi-compact separated \ncs\ and $\sF$ is tensor product or inverse image, see 
 Proposition~\ref{212}.
 \item  Miyatchi \citep{mi} proved that always
  $\cD^\si(\cA/\cC)\simeq\cD_\cC^\si(\cA)$, where $\si\in\set{+,-,b}$.
 \end{enumerate}
 \end{remark}

 We recall that a sequence $\row\cA m$ of triangulated subcategories of a triangulated category $\cA$ is said
 to be a \emph{semi-orthogonal decomposition} of $\cA$ if $\hom(A,A')=0$ for $A\in\cA_i,\,A'\in\cA_j$ and $i>j$,
 and for every object $A\in\cA$ there is a chain of morphisms
 \[
  0=A_m\xarr{f_m} A_{m-1}\xarr{f_{m-1}}\dots A_2\xarr{f_2} A_1\xarr{f_1} A_0=A
 \]
 such that $\con f_i\in\cA_i$ \citep{kl}. 

 \begin{corollary}\label{14} 
  Let $\sG:\cA\to\cB$ be an exact functor among triangulated categories, $\sF:\cB\to\cA$ be its right \lb left\rb\
  adjoint such that the natural morphism $\phi:\1_{\cB}\to\sG\sF$ \lb respectively, $\psi:\sG\sF\to\1_\cB$\rb\ is 
  an isomorphism. Then $(\im\sF,\ker\sG)$ \lb respectively, $(\ker\sG,\im\sF)$\rb\ is a semi-orthogonal
  decomposition of $\cA$.
 \end{corollary}
 \begin{proof}
  We consider the case of left adjoint. If $A=\sF B$ and $A'\in\ker\sG$, then 
  $\hom_\cA(A,A')\simeq\hom_\cB(\sG A,B)=0$. On the other hand, consider the natural morphism
  $f:\sF\sG A\to A$. Then $\sG f$ is an isomorhism, whence $\con f\in\ker\sG$. 
  So we can set $A_1=\sF\sG A$, $f_1=f$.
 \end{proof}

\section{Non-commutative schemes}
\label{s2}

\begin{definition}\label{21} 
\begin{enumerate}
\item  A \emph{\ncs} is a pair $(X,\kA)$, where $X$ is a scheme (called the \emph{commutative background}
 of the \ncs) and $\kA$ is a sheaf of $\kO_X$-algebras, which is \qch\ as a sheaf of $\kO_X$-modules. 
 Sometimes we say ``\ncs\ $\kA$'' not mentioning its commutative background $X$. We denote by $X\cl$
 the set of closed points of $X$.
 
\item  A \ncs\ $(X,\kA)$ is said to be \emph{affine} (\emph{separated, quasi-compact}) if so is its commutative 
 background $X$. It is said to be \emph{reduced} if $\kA$ has no nilpotent ideals.

\item A \emph{morphism} of \ncss\ $f:(Y,\kB)\to(X,\kA)$ is a pair $(f_X,f^\#)$, where $f_X:Y\to X$ is a
morphism of schemes and $f^\#$ is a morphism of $f_X^{-1}\kO_X$-algebras $f_X^{-1}\kA\to\kB$. In what
follows we usually write $f$ instead of $f_X$.

\item Given a \ncs\ $(X,\kA)$, we denote by $\kA\Md$ (respectively, by $\kA\md$) the category of \qch\ 
 (respectively, \cht) sheaves of $\kA$-modules. We call objects of this category just $\kA$-modules
 (respectively, \cht\ $\kA$-modules).
 
\item If $f:(Y,\kB)\to(X,\kA)$ is a morphism of \ncss, we denote by $f^*:\kA\Md\to\kB\Md$ the functor of
inverse image which maps an $\kA$-module $\kM$ to the $\kB$-module $\kB\*_{f^{-1}\kA}f^{-1}\kM$. If the map
$f_X$ is separated and quasi-compact, we denote by $f_*:\kB\Md\to\kA\Md$ the functor of direct image.
It follows from \cite[\S\,0.1 and \S\,1.9.2]{gro1} that these functors are well-defined. Moreover,
$f^*$ maps coherent modules to coherent ones.
\end{enumerate}

 In this paper we always suppose \ncss\ \emph{separated and quasi-compact}. 
 \end{definition} 
 
\begin{remark}\label{22}
 \begin{enumerate}
 \item  If $(X,\kA)$ is affine, i.e. $X=\spec\qR$ for some commutative ring $\qR$, then $\kA=\qA^\sim$
 is a sheafification of an $\qR$-algebra $\qA$. A \qch\ $\kA$-module is just a sheafification $M^\sim$ 
 of an $\qA$-module $M$, so $\kA\Md\simeq\qA\Md$ and we identify these categories. If, moreover, $\qA$ 
 is noetherian, then $\kA\md$ coincides with the category $\qA\md$ of finitely generated $\qA$-modules.
 
 \item  If $X$ is separated and quasi-compact, $\kA\Md$ is a \emph{Grothendieck category}. In particular,
 every \qch\ $\kA$-module has an injective envelope. We denote by $\kA\inj$ the full subcategory of $\kA\Md$
 consisting of injective modules.
 
 \item  The inverse image functor $f^*$ for a morphism of \ncss\ usually does not coincide with the inverse
 image functor $f_X^*$ with respect to the morphism of their commutative backgrounds. We can guarantee it
 if $\kB=f_X^*\kA$, for instance, if $Y$ is an open subset of $X$ and $\kB=\kA|_Y$.
 \end{enumerate}
\end{remark}
 
  \begin{definition}
  \begin{enumerate}
  \item   The \emph{center} of $\kA$ is the subsheaf $\cen\kA\sbe\kA$ such that 
 $$(\cen\kA)(U)=\setsuch{\al\in\kA(U)}{\al|_V\in\cen\kA(V) \text{ for all } V\sbe U},$$ 
 where $\cen\qA$ denotes the center of a ring $\qA$.   
 
 \item    We say that a \ncs\ $(X,\kA)$ is \emph{central}, if the natural homomorphism $\kO_X\to\kA$ maps $\kO_X$
 bijectively onto the center $\cen(\kA)$ of $\kA$.
  \end{enumerate}
 \end{definition}
 
 Note that if $(X,\kA)$ is affine, $X=\spec\qR$ and $\kA=\qA^\sim$, then $\cen\kA=(\cen\qA)^\sim$.
 
 \begin{proposition}\label{center} 
  $\End\1_{\kA\Md}\simeq\End\1_{\kA\inj}\simeq\Ga(X,\cen\kA)$.
 \end{proposition}
 \begin{proof}
  Let $\al\in\Ga(X,\cen\kA)$. Given any $\kM\in\kA\Md$, define $\al(\kM):\kM\to\kM$ by the rule:
  $\al(M)(U):\kM(U)\to\kM(U)$ is the multiplication by $\al|U$ for every open $U\sbe X$. 
  Obviously, it is a morphism of $\kA$-modules. Moreover, if $f\in\hom_\kA(\kM,\kN)$, one easily sees 
  that $f\al(M)=\al(N)f$, so $\al$ defines an element from $\End\1_{\kA\Md}$.
  
  Conversely, let $\la\in\End\1_{\kA\Md}$. Let $U\sbe X$ be an open subset, $j:U\to X$ be the embedding. 
  Then $\la(U)=\la(j_*j^*\kA)$ is an element from $\End_\kA(j_*j^*\kA)=\kA(U)$. Since $\la$ is an endomorphism of the identity functor,
  $\la(U)$ it is in $\cen\kA(U)$. Moreover, if $V\sbe U$ is another open subset, $j':V\to X$ is the embedding, the restriction
  homomorphism $r:j_*j^*\kA\to j'_*{j'}^*\kA$ gives the commutative diagram
  \[
   \xymatrix{ j_*j^*\kA \ar[r]^{\la(U)} \ar[d]^r & j_*j^*\kA \ar[d]^r \\
   			  j'_*{j'}^*\kA \ar[r]^{\la(V)} & j'_*{j'}^*\kA  }
  \]
  It implies that $\la(V)=\la(U)|V$. In particular, $\la(X)=\al$ is an element from $\Ga(X,\cen\kA)$ and
  $\la(U)$ coincides with the multiplication by $\al|U$. Thus we obtain an isomorphism 
  $\End\1_{\kA\Md}\simeq\Ga(X,\cen\kA)$.
  
  There is the restriction map $\End\1_{\kA\Md}\to\End\1_{\kA\inj}$. On the other hand, consider an injective
  copresentation of an $\kA$-module $\kM$, i.e. an exact sequence $0\to\kM\xarr{\al_\kM}\kI_\kM\to\kI'_\kM$ 
  with injective modules $\kI_\kM$ and $\kI'_\kM$. Let $\la\in\End\1_{\kA\inj}$. Then there is a unique homomorphism 
  $\la(\kM):\kM\to\kM$ such that $\la(\kI_\kM)\al_\kM=\al_\kM\la(\kM)$. Let $0\to\kN\xarr{\al_\kN}\kI_\kN\to\kI'_\kN$ 
  be an injective copresentation of another $\kA$-module $\kN$ and $f\in\hom_\kA(\kM,\kN)$. Extending $f$ to 
  injective copresentations, we obtain a commutative diagram
  \[
   \xymatrix{0 \ar[r] &\kM \ar[r]^{\al_\kM} \ar[d]^f & \kI_\kM \ar[r] \ar[d]^{f_0} & \kI'_\kM \ar[d]_{f_1} \\
   	0 \ar[r] &\kN \ar[r]^{\al_\kN}  & \kI_\kN \ar[r]  & \kI'_\kN  }
  \]
  It implies that
  \begin{multline*}
   \al_\kN\la(\kN)f=\la(\kI_\kN)\al_\kN f=\la(\kI_\kN)f_0\al_\kM=\\=f_0\la(\kI_\kM)\al_\kM
   		=f_0\al_\kM\la(\kM)=\al_\kN f\la(\kM),
  \end{multline*}
  whence $\la(\kN)f=f\la(\kM)$, so we have extended $\la$ to a unique endomorphism of $\1_{\kA\Md}$.
 \end{proof}
 
  \begin{proposition}\label{23} 
 Let $\kC=\cen(\kA)$, $X'=\spec\kC$ be the spectrum of the \lb commutative\rb\ $\kO_X$-algebra
 $\kC$, $\phi:X'\to X$ be the structural morphism, and $\kA'=\phi^{-1}\kA$. 
\begin{enumerate}
 \item $\kA'$ is an $\kO_{X'}$-algebra, so $(X',\kA')$ is a central \ncs. 
 
 \item For any $\kF\in\kA\Md$ the natural map $\kF\to\phi_*\phi^*\kF$ is an isomorphism.\!%
 \footnote{\,Note that in this situation $\phi^*=\phi^{-1}$.}

 \item For any $\kF'\in\kA'\Md$ the natural map $\phi^*\phi_*\kF'\to\kF'$ is an isomorphism.
 
 \item The functors $\phi^*$ and $\phi_*$ establish an equivalence of the
 categories $\kA\Md$ and $\kA'\Md$ as well as of $\kA\md$ and $\kA'\md$.
 \end{enumerate}   
 \emph{Thus, when necessary, we can suppose, without loss of generality, that our \ncss\ are central.}
 \end{proposition}
 \begin{proof}
  All claims are obviously local, so we can suppose that $X=\spec\qR$ and $X'=\spec\qR'$, where $\qR'$
  is the center of the $\qR$-algebra $\qA=\Ga(X,\kA)$. Then all claims are trivial.
 \end{proof}

 We call a \ncs\ $(X,\kA)$ \emph{noetherian} if the scheme $X$ is noetherian and $\kA$ is
 \cht\ as a sheaf of $\kO_X$-modules. Note that if $(X,\kA)$ is noetherian, the central 
 \ncs\ $(X',\kA')$ constructed in Proposition~\ref{23} is also noetherian. Indeed, if an 
 affine \ncs\ $(\spec\qR,\qA\til)$ is noetherian, then $\qA$ is a \emph{noetherian algebra}, i.e. 
 $\qC=\cen\qA$ is noetherian and $\qA$ is a finitely generated $\qC$-module. 
 
 \begin{definition}\label{24} 
  Let $(X,\kA)$ be noetherian.
\begin{enumerate}
  \item   We denote by $\lp\kA$ the full subcategory of $\kA\md$ consisting of
  \emph{\lop} modules $\kP$, i.e. such that $\kP_x$ is a projective $\kA_x$-module for every $x\in X$.
  
  \item We say that $\kA$ \emph{has enough \lop\ modules} if for every \cht\ $\kA$-module $\kM$ there is
  an epimorphism $\kP\to\kM$, where $\kP\in\lp\kA$. Since every \qch\ module is a sum of its \cht\
  submodules, then for every \qch\ $\kA$-module $\kM$ there is an epmorphism $\kP\to\kM$, where $\kP$
  is a coproduct of modules from $\lp\kA$.
\end{enumerate}    
\end{definition}

 An important example arises as follows. We say that a noetherian \ncs\ $(X,\kA)$ is \emph{quasi-projective}
 if there is an ample $\kO_X$-module $\kL$ \cite[Section~4.5]{gro2}. Note that in this case $X$ is indeed 
 a quasi-projective scheme over the ring $R=\bop_{n=0}^\8\Ga(X,\kL^{\*n}$).
 
 \begin{proposition}\label{25} 
  Every quasi-projective \ncs\ $(X,\kA)$ has enough \lop\ modules.
 \end{proposition}
  \begin{proof}
 Let $\kL$ be an ample $\kO_X$-module, $\kM$ be any \cht\ $\kA$-module. There is an epimorphism
 of $\kO_X$-modules $n\kO_X\to\kM\*_{\kO_X}\kL^{\*m}$ for some $m$, hence also an epimorphism
 $\kF=n\kL^{\*(-m)}\to\kM$. Since $\hom_\kA(\kA\*_{\kO_X}\kF,\kM)\simeq\hom_{\kO_X}(\kF,\kM)$, 
 it gives an epimorphism of $\kA$-modules $\kA\*_{\kO_X}\kF\to\kM$, where $\kA\*_{\kO_X}\kF\in\lp\kA$.
 \end{proof}
 
 We define an \emph{invertible $\kA$-module} as an $\kA$-module $\kI$ such that $\End_\kA\kI\simeq\kA\op$
 and the natural map $\Hom_\kA(\kI,\kA)\*_\kA\kI\to(\End_\kA\kI)\op\simeq\kA$ is an isomorphism.
 For instance, the modules constructed in the preceding proof are direct sums of invertible modules. On the contrary, one
 easily proves that, if $\kA$ is noetherian and $\cen\kA=\kO_X$, any invertible $\kA$-module
 $\kI$ is isomorphic to $\kA\*_{\kO_X}\kL$, where $\kL=\Hom_{\kA\mt\kA}(\kI,\kI)$ and $\kL$ is an
 invertible $\kO_X$-module. (We will not use this fact.)
 
\medskip
 We denote by $\cC\kA$ the category of complexes of $\kA$-modules, by $\cH\kA$ the category of complexes modulo homotopy
 and by $\cD\kA$ the derived category $\cD(\kA\Md)$. We also use the conventional notations $\cC^\si\kA$ $\cH^\si\kA$ 
 and $\cD^\si\kA$, where $\si\in\set{+,-,b}$. We denote by $\cD\sm\kA$ the full subcategory of \emph{compact} objects
 $\kC\sd$ from $\cD\kA$, i.e. such that the natural morphism 
 \[
 \cop_i\hom_{\cD\kA}(\kC\sd,\kF\sd_i)\to\hom_{\cD\kA}(\kC\sd,\cop_i\kF\sd_i)
 \]
 is bijective for any coproduct $\cop_i\kF\sd_i$. 

  Recall that a complex $\kI\sd$ is said to be \emph{K-injective} \citep{sp} if for every acyclic complex
 $\kC\sd$ the complex $\hom\sd(\kC\sd,\kI\sd)$ is acyclic too. We denote by $\kinj\kA$
 the full subcategory of $\cH\kA$ consisting of K-injective complexes and by $\kinj_0\kA$ its
 full subcategory consisting of acyclic K-injective complexes.
 
 \begin{proposition}\label{28} 
 Let $(X,\kA)$ be a \ncs\ \lb separated and quasi-compact\rb.
 \begin{enumerate}
 \item  For every complex $\kC\sd$ in $\cC\kA$ there is a \emph{K-injective resolution}, i.e. 
 a K-injective complex $\kI\sd\in\cC\kA$ together with a \qis\ $\kC\sd\to\kI\sd$.
 
 \item  $\cD\kA\simeq\kinj\kA/\kinj_0\kA$.
 \end{enumerate} 
 \end{proposition}
 \begin{proof}
   As the category $\kA\Md$ is a Grothendieck category, (1) follows immediately from 
 \cite[Theorem~5.4]{ajs} (see also \cite[Lemma 3.7 and Proposition 3.13]{sp}).
 Then (2) follows from \cite[Proposition~1.5]{sp}.
 \end{proof} 
 
 A complex $\kF\sd$ is said to be \emph{K-flat} \citep{sp} if for every acyclic complex $\kS\sd$ 
 of \emph{right} $\kA$-modules the complex $\kF\sd\*_\kA\kS\sd$ is acyclic. The next result is quite 
 analogous to \cite[Proposition~1.1]{ajl} and the proof just repeats that of the cited paper with no
 changes.
 
 \begin{proposition}\label{29} 
 Let $(X,\kA)$ be a \ncs. Then for every complex $\kC\sd$ in $\cC\kA$ there is a
 \emph{K-flat replica}, i.e. a K-flat complex $\kF\sd$ quasi-isomorphic to $\kC\sd$.
 \end{proposition}
 
 \begin{remark}\label{210} 
 If $(X,\kA)$ is noetherian and has enough \lop\ modules, every complex from $\cC^-\kA$ has 
 a \lop\ (hence flat) resolution. Then \cite[Theorem~3.4]{sp} implies that for every complex 
 $\kC$ from $\cC\kA$ there is an \emph{Lp-resolution}, i.e. a K-flat complex
 $\kF\sd$ consisting of \lop\ modules together with a \qis\ $\kF\sd\to\kC\sd$. 
 For instance, it is the case if $(X,\kA)$ is \emph{quasi-projective} (Proposition~\ref{25}).
 \end{remark}

\smallskip
 A complex $\kI\sd$ is said to be \emph{weakly K-injective} if for every acyclic K-flat complex
 $\kF\sd$ the complex $\hom\sd(\kF\sd,\kI\sd)$ is exact. 
 
 \begin{proposition}{\rm\citep[Propositions~5.4~and~5.15]{sp}}\label{211} \
 Let $f:(X,\kA)\to(Y,\kB)$ be a morphism of \ncs.
 \begin{enumerate}
  \item  If $\kF\sd\in\cC\kB$ is K-flat, then so is also
 $f^*\kF\sd$. If, moreover, $\kF\sd$ is K-flat and acyclic, then $f^*\kF\sd$ is acyclic too.
 
 \item  If $\kI\in\cC\kA$ is weakly K-injective, then $f_*\kI$ is weakly K-injective. If, moreover,
 $\kI$ is weakly K-injective and acyclic, then $f_*\kI$ is acyclic too.
  \end{enumerate} 
 \end{proposition}
 
 \begin{proposition}{\rm\citep[Section~6]{sp}}\label{212} \
 Let $(X,\kA)$ be a \ncs.
 \begin{enumerate}
 \item  The derived functors $\rhom_\kA\sd(\kF\sd,\kG\sd)$ 
 and $\rHom_\kA\sd(\kF\sd,\kG\sd)$ exist and can be calculated using a K-injective resolution
 of $\kG\sd$ or a weakly K-injective resolution of $\kG\sd$ and a K-flat replica of $\kF\sd$.
 
 \item The derived functor $\kF\sd\tl_\kA\kG\sd$, where $\kG\sd\in\cD\kA\op$,
 exists and can be calculated using a K-flat replica either of $\kF\sd$ or of $\kG\sd$.
 Moreover, if $\kG\sd$ is a complex of $\kA\mt\kB$-bimodules, where $\kB$ is another sheaf of 
 $\kO_X$-algebras, there are isomorphisms of functors
 \begin{align*}
  \rhom_\kB(\kF\sd\tl_\kA\kG\sd,\kM\sd)&\simeq \rhom_\kA(\kF\sd,\rHom_\kB(\kG\sd,\kM\sd))\\
  \rHom_\kB(\kF\sd\tl_\kA\kG\sd,\kM\sd)&\simeq \rHom_\kA(\kF\sd,\rHom_\kB(\kG\sd,\kM\sd)). 
 \end{align*}
 
 \item For every morphism $f:(X,\kA)\to(Y,\kB)$ the derived functors $\rL f^*:\cD\kB\to\cD\kA$ and
 $\rR f_*:\cD\kA\to\cD\kB$ exist. They can be calculated using, respectively,
 K-flat replicas in $\cC\kB$ and weakly K-injective resolutions in $\cC\kA$.
  Moreover, there are isomorphisms of functors
 \begin{align*}
   \rhom\sd_\kB(\kF\sd,\rR f_*\kG\sd)&\simeq\rhom\sd_\kA(\rL f^*\kF\sd,\kG\sd)\\
  \rHom\sd_\kB(\kF\sd,\rR f_*\kG\sd)&\simeq\rR f_*\rHom\sd_\kA(\rL f^*\kF\sd,\kG\sd).  
 \end{align*}
 
 \item If $g:(Y,\kB)\to(Z,\kC)$ is another morphism of \ncss, then 
 $\rL(gf)^*\simeq\rL f^*\cir\rL g^*$ and
 $\rR(gf)_*\simeq\rR g_*\cir\rR f_*$.
 \end{enumerate}
 If the considered \ncss\ have enough \lop\ modules \lb for instance, 
 are quasi-projective\rb, one can replace in these statements K-flat replicas by Lp-resolutions.
 \end{proposition}
 
 In particular, let $f:\qA\to\qB$ be a homomorphism of rings. We consider $\qB$ as an algebra over
 a subring $\qS$ (an arbitrary one) of its center and $\qA$ as an algebra over a subring 
 $\qR\sbe\cen\qA\cap f^{-1}(\qS)$. Then we can identify $f$ with its sheafification
 $f\til:(\spec\qS,\qB\til)\to(\spec\qR,\qA\til)$. In this context the functors $(f\til)^*$ and
 $(f\til)_*$ are just sheafifications, respectively, of the ``back-up'' functor ${_BM}\mps{_AM}$ 
 and the ``change-of-scalars'' functor ${_AN}\mps {_BB\*_AN}$. 
 
 \begin{definition}\label{perf-def} 
  A complex $\kC\sd$ in $\cC\kA$ is said to be \emph{perfect} if for every point $x\in X$ there is an open
  neighbourhood $U$ of $x$ such that $\kC|_U$ is quasi-isomorphic to a finite complex of \lop\ coherent modules.
  We denote by $\per\kA$ the full subcategory of $\cD\kA$ consisting of perfect complexes.
 \end{definition} 
 
 The following result is well-known in commutative and affine cases \citep{ne2,rou}. Though the proof in
 non-commutative situation is almost the same, we include it for the sake of completeness. Actually, we reproduce
 the proof of Rouquier with slight changes.
 
 \begin{theorem}\label{perf} 
 Let $(X,\kA)$ be a \ncs\ \lb quasi-compact and separated\rb. Then $\cD\kA$ is compactly generated and 
 $\cD\sm\kA=\per\kA$.
 \end{theorem}
 \begin{proof}
 Let $U\sbe X$ be an open affine subset of $X$, $\kA_U$ be the restriction of $\kA$ onto $U$, $\comp U=X\=U$, 
 $j=j_U:U\to X$ be the embedding. Then the inverse image functor $j^*:\kA\Md\to\kA_U\Md$ is exact and the natural 
 morphism $j^*j_*\to\1_{\kA_U\Md}$ is an isomorphism (actually, identity). Therefore $\ker j^*$ is a localizing 
 subcategory and $\kA_U\Md\simeq\kA\Md/\ker j^*$. Note that $\ker j_U^*$ consists of the $\kA$-modules $\kM$ such
 that $\supp\kM\sbe\comp U$. Then $\ker\rL j_*$ is a localizing subcategory of $\cD\kA$ and 
 $\cD\kA_U\simeq\cD\kA/\ker\rL j_*$. This kernel coincide with the full subcategory $\cD_{\comp U}\kA$ of
 $\cD\kA$ consisting of complexes whose cohomologies are supported on $\comp U$.
 
 If $W\sbe X$ is another open affine subset, then the subcategories $\cD_{\comp U}\kA$ and $\cD_{\comp W}\kA$ 
 \emph{intersect properly} in the sense of \cite[5.2.3]{rou}. Recall that it means that $j_W^*{j_U}_*j_U^*\kF=0$
 as soon as $j_W^*\kF=0$, what follows, for instance, from \cite[Corollaire~(1.5.2)]{gro2} applied to the cartesian 
 diagram of affine morphisms (open embeddings)
 \[
  \begin{CD}
   U\cap W @>{j'_W}>> U \\
   @V{j'_U}VV  @VV{j_U}V\\
   W @>>{j_W}> X
  \end{CD}
 \]
  Therefore, if
 $X=\bup_{i=1}^mU_i$ is an open affine covering of $X$, then $\set{\cD_{\comp U_i}\kA}$ is a cocovering of the 
 triangulated category $\cD\kA$ as defined in \cite[5.3.3]{rou}. If $S\sb\set{1,2,\dots,m}$ does not contain
 $i$, $U_S=\bup_{j\in S}U_j$, then $\bap_{j\in S}\cD_{\comp U_j}\kA=\cD_{\comp U_S}\kA$ and the image of
 $\cD_{\comp U_S}\kA$ in $\cD\kA_{U_i}$ coincides with $\cD_{U_i\=U_S}\kA_{U_i}$. There are sections 
 $\lst fk\in\qA=\Ga(U_i,\kO_X)$ such that $U_i\=U_S=V\row fk$ as a closed subset of $U_i$. The following lemma
 shows that the subcategory $\cD_{U_i\=U_S}\kA_{U_i}$ is compactly generated in $\cD\kA_{U_i}$.
 
 \begin{lemma}\label{koszul} 
  Let $\qA$ be an algebra over a commutative ring $\qO$ and $\qI=\row fk$ be a finitely generated ideal in
  $\qO$. Let $K\sd(\qI)$ be the corresponding Koszul complex. Denote by $\qA\Md_\qI$ the full subcategory 
  of $\qA\Md$ consisting of all modules $M$ such that for every element $a\in M$ there is $m$ such that 
  $\qI^ma=0$. Denote by $\cD_\qI\qA$ the full subcategory of $\cD\qA$ consisting of all complexes such that 
  their cohomologies belong to $\qA\Md_\qI$. Then $\cD_\qI\kA$ is generated by the complex 
  $K_\qA\sd(\qI)=\qA\*_\qO K\sd(\qI)$.
 \end{lemma}
 \begin{proof}
  Note that $\hom_{\cD\qA}(K_\qA\sd(\qI),C\sd)\simeq\hom_{\cD\qO}(K\sd(\qI),C\sd)$ for every $C\sd\in\cD\qA$. 
  If $C\sd\in\cD_\qI\qA$ is non-exact, then $\hom_{\cD\qO}(K\sd(\qI),C\sd[n])\ne0$ for some $n$ by 
  \cite[Proposition~6.6]{rou}. It proves the claim.
 \end{proof}
 
 Evidently, $K\sd_\qA(\qI)$ is compact in $\cD\qA$. So we can now use \cite[Theorem~5.15]{rou}. It implies 
 that $\cD\kA$ is compactly generated and a complex $\kC\sd$ in $\cD\kA$ is compact \iff $j_{U_i}^*\kC\sd$ 
 is compact in $\cD\kA_{U_i}$ for every $1\le i\le m$. As $U_i$ is affine, compact complexes in $\cD\kA_{U_i}$ 
 are just perfect complexes. Therefore, it is true for $\cD\kA$ too.
 \end{proof}

\section{Minors}
\label{s3} 

  \begin{definition}\label{31} 
  Let $(X,\kB)$ be a \ncs, $\kP$ be a \flop\ $\kB$-module, 
  $\kA=(\END_\kB\kP)\op$. The \ncs\ $(X,\kA)$ is called a \emph{minor} of the \ncs\ $(X,\kB)$.\!%
 \footnote{\,In the affine case this notion was introduced in \citep{minor}. Actually, the main results
 of this section are just global analogues of those from \citep{minor}.} 
  \end{definition} 
 
 In this situation we consider $\kP$ as $\kB\mt\kA$-bimodule (left over $\kB$, right over $\kA$).
 Let $\kP\ch=\Hom_\kB(\kP,\kB)$. It is an $\kA\mt\kB$-bimodule, \flop\
 over $\kB$. The following statements are evidently local, then they are well-known.

 \begin{proposition}\label{32} 
 The natural homomorphism $\kP\to\Hom_\kB(\kP\ch,\kB)$ is an isomorphism. Moreover, 
 $\kA\simeq\END_\kB\kP\ch\simeq \kP\ch\*_\kB\kP$.
 \end{proposition}

 We consider the following functors:
\begin{equation}\label{fgh} 
\begin{split}
 \sF=\kP\*_\kA\_&: \kA\Md\to\kB\Md,\\
 \sG=\Hom_\kB(\kP,\_)&: \kB\Md\to\kA\Md,\\
 \sH=\Hom_\kA(\kP\ch,\_)&: \kA\Md\to\kB\Md.
\end{split} 
\end{equation}
 Note that $\sG$ is exact and $\sG\simeq\kP\ch\*_\kB\_\,$, so both $(\sF,\sG)$ and $(\sG,\sH)$ 
 are adjoint pairs of functors. If the \ncs\ $(X,\kB)$ is noetherian, so is also $(X,\kA)$ 
 and these functors map \cht\ sheaves to \cht\ ones. 
 
 We set $\kI_\kP=\im\{\mu_\kP:\kP\*_\kA\kP\ch\to\kB\}$, where $\mu(p\*\ga)=\ga(p)$.
 
\begin{theorem}\label{FG} 
\begin{enumerate}
 \item  $\sG$ is a bilocalization functor, thus $\cC$ is a bilocalizing subcategory, $\kA\Md\simeq\kB\Md/\cC$, where
$\cC=\ker\sF=\kP^\perp$ and both $\sF$ and $\sH$ are full embeddings $\kA\Md\to\kB\Md$ (usually with different images).

 \item $\cC=\setsuch{\kM\in\kB\Md}{\kI_\kP\kM=0}\simeq(\kB/\kI_\kP)\Md.$
 
\item  $\im\sF={^\perp\cC}$ coincides with the full subcategory of $\kB\Md$ consisting 
 of all modules $\kM$ such that for every point $x\in X$ there is an exact sequence 
 $P_1\to P_0\to \kM_x\to 0$, where $P_0,P_1$ are multiples of $\kP_x$ \lb i.e. direct sums, maybe infinite, 
 of its copies\rb. We denote this subcategory by $\kP\Md$.

\item[3$'$.] $\im\sH=\cC^\perp$ coinsides with the full subcategory of $\kB\Md$ consisting of all modules $\kM$ such that
 there is an exact sequence $0\to\kM\to\kI_0\to\kI_1$, where $\kI_i\in\sH(\kA\inj)$.\!%
 \footnote{\,Note that all $\kB$-modules from $\sH(\kA\inj)$ are injective.} 
 We denote this subcategory by $\kP\Inj\Md$.
\end{enumerate}
\end{theorem}
\begin{proof}
 Theorem~\ref{11} and Corollary~\ref{biloc} show that, to prove claims \textit{1,3} and \textit{3$'$}, it is enough to prove 
 the following statements.
 
 \begin{proposition}\label{FG0} 
 \begin{enumerate}
 \item The natural morphism $\phi:\1_{\kA\Md}\to\sG\cir\sF$ is an isomorphism.
 \item $\im\sF=\kP\Md$.
 \item[2$'$.] $\im\sH=\kP\Inj\Md$.
 \end{enumerate}
 \end{proposition}
 \begin{proof}
 As the claims \textit{1} and \textit{2} are local, we can suppose that the \ncs\ $(X,\kB)$ is affine, so replace $\kB\Md$
 by $\qB\Md$, where $\qB=\Ga(X,\kB)$. Then $\kP=P\til$ for some finitely generated projective
 $\qB$-module and $\kA=\qA\til$, where $\boldsymbol A=(\End_{\boldsymbol B}P)\op$. Hence we can also replace 
 $\kA\Md$ by $\qA\Md$ and $\kP\Md$ by $P\Md$, the full subcategory of $\qB\Md$ consisting of all modules 
 $N$ such that there is an exact sequence $P_1\to P_0\to N\to0$, where $P_i$ are multiples of $P$. 
 
 Obviously, $\phi(\qA)$ is an isomorphism. Since $\sF$ and $\sG$ preserve
 arbitrary coproducts, $\phi(F)$ is an isomorphism for any free $\qA$-module $F$. 
 Let $M\in\kA\Md$. There is an exact sequence $F_1\to F_0\to M\to 0$, where $F_0,F_1$ are free modules,
 which gives rise to a commutative diagram with exact rows
 \[
 \xymatrix@C=1.5em{ F_1 \ar[r] \ar[d]_{\phi(F_1)} & F_0 \ar[r] \ar[d]_{\phi(F_0)} 
 & M \ar[r] \ar[d]_{\phi(M)}& 0 \\
   \sG\cir\sF(F_1) \ar[r] &  \sG\cir\sF(F_0) \ar[r] &  \sG\cir\sF(M) \ar[r] & 0
   }
 \]
 As the first two vertical arrows are isomorphisms, so is $\phi(M)$, which proves claim \textit{1}.
 Moreover, we get an exact sequence $\sF(F_1)\to\sF(F_0)\to\sF(M)\to0$, where $\sF(F_i)$ are
 multiples of $\sF(\qA)=P$. Therefore, $\sF(M)\in P\Md$. 
 
 Consider now the natural morphism $\psi:\sF\cir\sG\to\1_{B\Md}$. This time $\psi(P)$ is an isomorphism.
 Let now $N$ be a $\qB$-module such that there is an exact sequence $P_1\to P_0\to N\to0$, where $P_i$
 are multiples of $P$. Then there is a commutative diagram with exact rows
 \[
 \xymatrix@C=1.5em{ \sF\cir\sG(P_1) \ar[r] \ar[d]_{\psi(P_1)} 
 & \sF\cir\sG(P_0) \ar[r] \ar[d]_{\psi(P_0)} 
 & \sF\cir\sG(N) \ar[r] \ar[d]_{\psi(N)}& 0 \\
   P_1 \ar[r] &  P_0 \ar[r] & N \ar[r] & 0
   }
 \]
 The first two vertical arrows are isomorphisms, so $\psi(N)$ is also an isomorphism. 
 It proves claim \textit{3}.
 
 The proof of \textit{3$'$} is quite analogous to that of \textit{3}, so we omit it. 
 
 Note that the condition $\kM\in\kP\Inj\Md$ also turns out to be local, since it means that the natural map
 $\kM\to\sH\cir\sG(\kM)$ is an isomorphism.
 \end{proof}

 The statement \textit{2} is also local, so we only have to prove it for a ring $\qB$,
 a finitely generated projective $\qB$-module $P$ and the ideal $I_P=\im\mu_P$.
  It follows from \cite[Proposition VII.3.1]{ce} that $I_PP=P$. Therefore,
  if $f:P\to M$ is non-zero, then $I_P\im f=\im f\ne0$, hence $I_PM\ne0$. On the contrary,
  if $I_PM\ne0$, there is an element $u\in M$, elements $p_i\in P$ and homomorphisms
  $\ga_i:P\to B$ such that $\sum_i\ga_i(p_i)u\ne0$. 
  Let $\be:B\to M$ maps $1$ to $u$ and $\ga_i^u=\be\ga_i$. Then at least one of the
  homomorpisms $\ga^u_i$ is non-zero.
 \end{proof}
 
 The functor $\sG$ is exact, so it induces a functor $\DG:\cD\kB\to\cD\kA$ mapping a complex $\kF\sd$
 to $\sG\kF\sd$. It is both left and right derived functor of $\sG$. We can also consider the left
 derived functor $\LF$ of $\sF$ and the right derived functor $\RH$ of $\sH$, 
 both being functors $\cD\kA\to\cD\kB$.
 Obviously, $\DG$ maps $\cD^\si\kB$ to $\cD^\si\kA$, where $\si\in\set{+,-,b}$,
 $\LF$ maps $\cD^-\kA$ to $\cD^-\kB$ and $\RH$ maps
 $\cD^+\kA$ to $\cD^+\kB$.
 
 \begin{theorem}\label{DFG} 
 \begin{enumerate}
 \item  The functors $(\LF,\DG)$ and $(\DG,\RH)$ form adjoint pairs. 
 
 \item $\DG$ is a bilocalization functor, $\ker\DG=\cD_\cC\kB$, where $\cC=\ker\sG$ is a 
 bilocalizing subcategory, $\cD\kA\simeq\cD\kB/\cD_\cC\kB$ and both $\LF$ and $\RH$ are full embeddings $\cD\kA\to\cD\kB$ 
 (usually with different images).
 
 \item  The functor $\LF$ maps $\cD\sm\kA$ to $\cD\sm\kB$. 
 
 \item  $(\ker\DG,\im\LF)$ as well as $(\im\RH,\ker\DG)$ are semi-orthogonal decompositions of
 $\cD\kB$.
 
 \item  $\im\LF={^\perp(\cD_\cC\kB)}$ coincides with the full subcategory $\cD\ukP$ of $\cD\kB$ 
 consisting of complexes quasi-isomorphic to K-flat complexes $\kF\sd$ such that for 
 every $x\in X$ and every component $\kF^i$ the localization $\kF^i_x$ is a direct limit of
 modules from $\add\kP_x$. The same is true if we replace $\cD$ by $\cD^-$.
 
 \item[5p.]   If $\kA$ and $\kB$ have enough \lop\ modules \lb for instance, if $X$ is quasi-projective\rb, 
 $\im\LF$ coincides with the full subcategory $\cD\kP$ of $\cD\kB$ consisting of complexes 
 quasi-isomorphic to K-flat complexes $\kF\sd$ such that $\kF^i_x\in\Add\kP_x$ for every 
 $i\in\mZ$ and every point $x\in X$. The same is true if we replace $\cD$ by $\cD^-$.
 
 \item[5$'$.] $\im\RH=(\cD_\cC\kB)^\perp$ coincides with the full subcategory $\cD\kP\Inj$ of $\cD\kB$ consisting of
 complexes quasi-isomorphic to K-injective complexes consisting of modules from $\sH(\kA\inj)$. The same
 is true if we replace $\cD$ by $\cD^+$.
\end{enumerate} 
 \emph{Note that the condition in 5$'$ can also be verified locally at every point $x\in X$.}
 \end{theorem}
 \begin{proof}
  \textit{1.} Let $\kF\sd$ be a K-flat replica of $\kM\sd\in\cD\kA$ and $\kI\sd$ be an injective resolution
  of $\kN\sd\in\cD\kB$. Then $\LF\kM\sd=\sF\kF\sd$ and $\DG\kN\sd=\sG\kI\sd$. As $\kP\in\lp\kB$, the
  complex $\sF\kF\sd$ is K-flat and the complex $\sG\kI\sd$ is K-injective. 
  By Proposition~\ref{212}\,\textit{2},
 \begin{multline*}
  \rhom_\kB(\sF\kF\sd,\kI\sd)=\hom_\kB\sd(\sF\kF\sd,\kI\sd)\simeq\\ 
  \hom_\kA\sd(\kF\sd,\sG\kI\sd)=\rhom_\kA(\kF\sd,\sG\kI\sd)  .
 \end{multline*} Taking zero cohomologies, we obtain that
 \[
  \hom_\kB(\sF\kF\sd,\kI\sd)\simeq \hom_\kA(\kF\sd,\sG\kI\sd).
 \]
 %
  Choose now a K-flat replica $\kG\sd$ of $\kN\sd$ and a K-injective resolution $\kJ\sd$ of $\kM\sd$.
  Then $\DG\kN\sd=\sG\kG\sd$ and $\RH\kM\sd=\sH\kJ\sd$. By \cite[Proposition~5.14]{sp}, $\sH\kJ\sd$
  is weakly K-injective. By Proposition~\ref{212}\,\textit{2} and \cite[Proposition~6.1]{sp},
 \begin{multline*}
   \rhom_\kA(\sG\kG\sd,\kJ\sd)=\hom_\kA\sd(\sG\kG\sd,\kJ\sd)\simeq \\
   \hom_\kB\sd(\kG\sd,\sH\kJ\sd)= \rhom_\kB(\kG\sd,\sH\kJ\sd).
 \end{multline*} 
 Taking zero cohomologies, we obtain that
 \[
  \hom_\kA(\sG\kG\sd,\kJ\sd)\simeq \hom_\kB(\kG\sd,\sH\kJ\sd)
 \]
 
 \smallskip
  \textit{2} follows now from Theorem~\ref{FG} and Theorem~\ref{11}. The statement \textit{4}
  follows from Theorem~\ref{14}. 
  
  \smallskip
  \textit{3.} As the right adjoint $\DG$ of $\LF$ preserves arbitrary coproducts, $\LF$ 
  maps compact objects to compact ones. 
  
  \smallskip
  \textit{4\,} follows now from Corollary~\ref{14}.
  
  \smallskip
  \textit{5.} The construction of \cite[Proposition~1.1]{ajl} gives for any complex $\kM\sd\in\cD\cA$
  a quasi-isomorphic K-flat complex $\kF\sd$ such that all its components $\kF^i$ are flat. 
  Moreover, $\kF\sd$ is left bounded if so is $\kM\sd$.
  By \cite[Ch.\,X,\,\S\,1,\,Th\'eor\`eme~1]{boh}, $\kF^i_x\simeq\dlim\kL^i_n$, where $\kL^i_n$ 
  are projective finitely generated $\kA_x$-modules, hence belong to $\add\kA_x$. Then 
  $\LF\kM\sd\simeq\sF\kF\sd$. As $\sF$ preserves direct limits and $\sF\kA\simeq\kP$, 
  $\sF\kF^i_x\simeq\dlim\sF\kL^i_n$ and $\sF\kL^i_n\in\add\kP_x$. Hence $\kM\sd\in\cD\ukP$.
  
  On the contrary, let $\kN\sd\in\cD\ukP$. We can suppose that it is K-flat and for every 
  $i\in\mZ$ and every $x\in X$ we can present $\kN^i_x $ as $\dlim\kN^i_n$, where 
  $\kN^i_n\in\add\kP_x$. Then the complex $\sG\kN\sd$ is also K-flat \cite[Proposition~5.4]{sp},
  so $\LF\cir\DG(\kN\sd)\simeq\sF\sG(\kN\sd)$. As the natural map $\sF\sG(\kP)\to\kP$ is 
  an isomorphism, the same is true for all modules $\kN^i_n$, hence also for $\kN^i_x$. 
  Therefore, the natural map $\LF\cir\DG(\kN)\to\kN$ is an isomorphism. 

  The proof of \textit{5p} is quite analogous to the proof of \textit{5}, taking into account that in this
  situation every complex is quasi-isomorphic to a K-flat complex of \lop\ modules. The proof of \textit{5$'$}
  is also analogous to that of \textit{5}.
 \end{proof}
 
   Recall that $\cC=\ker\sG\simeq(\kB/\kI_\kP)\Md$.
 There is one special case when the category $\ker\DG$ can be described analogously.
 
 \begin{theorem}\label{qmod} 
  Suppose that the ideal $\kI_\kP$ is flat as a right $\kB$-module. Then $\ker\DG\simeq\cD(\kB/\kI_\kP)$.
 \end{theorem}
 \begin{proof}
  Let $\kI=\kI_\kP$, $\kQ=\kB/\kI$. One easily sees that $\kI^2=\kI$. We identify $\cD\kQ$ with the full 
  triangulated subcategory of $\cD\kB$, obviously contained in $\ker\DG$. Let $\kF\sd\in\ker\DG$, i.e. 
  its cohomologies are indeed $\kQ$-modules. We can suppose that $\kF\sd$ is K-flat. Tensoring it with the 
  exact sequence $0\to\kI\to\kB\to\kQ\to0$, we obtain an exact sequence of complexes
  $0\to\kI\*_\kB\kF\sd\to\kF\sd\to\kQ\*_\kB\kF\sd\to0$. Since $\kI$ is flat,
  $H\sd(\kI\*_\kB\kF\sd)\simeq\kI\*_\kB H\sd(\kF\sd)$. Note that 
  $\kI\*_\kB\kQ\simeq\kI/\kI^2=0$, whence $\kI\*_\kB\kM=0$ for any $\kQ$-module. Therefore,
  $H\sd(\kI\*_\kB\kF\sd)=0$, hence $\kF\sd$ is quasi-isomorphic to $\kQ\*_\kB\kF\sd$, which is in $\cD\kQ$.
 \end{proof}

 \begin{example}\label{endo} 
 An important special case of minors appears as the \emph{endomorphism construction}.
 Let $\kA$ be a \ncs, $\kF$ be a \cht\ $\kA$-module and $\kA_\kF=\END_\kA(\kA\+\kF)\op$. Then $\kA_\kF$
 is identified with the algebra of matrices
 \[
  \kA_\kF=\mtr{\kA&\kF\\\kF'&\kE}
 \]
 where $\kF'=\Hom_\kA(\kF,\kA)$ and $\kE=(\END_\kA\kF)\op$. If $\kP_\kF=\smtr{\kA\\\kF'}$ considered
 as $\kA_\kF$-module, then $\kA\simeq(\END_{\kA_\kF}\kP_\kF)\op$, so $\kA$ is a minor of $\kA_\kF$ 
 and the categories $\kA\Md$ and $\cD\kA$ are bilocalizations, respectively, of $\kA_\kF\Md$ and 
 $\cD\kA_\kF$. The corresponding functors are
 \begin{align*}
  \sF_\kF&=\kP_\kF\*_\kA\_\,,\\
  \sG_\kF&=\Hom_{\kA_\kF}(\kP_\kF,\_\,),\\
  \sH_\kF&=\Hom_{\kA_\kF}(\kP_\kF,\_\,).
 \end{align*}
 Note that $\kP_\kF\ch\simeq\mtr{\kA&\kF}$ as right $\kA_\kF$-module and, by the construction, 
 $\kP_\kF\simeq\Hom_\kA(\kP_\kF\ch,\kA)$. Theorem~\ref{FG}.\textit{2} then implies that the kernel $\cC$ of 
 the functor $\sG_\kF:\kA_\kF\Md\to\kA\Md$ is equivalent to $\kE/\kI_\kF\Md$, where 
 $\kI_\kF$ is the image of the natural map $\kF'\*_\kA\kF\to\kE$. This construction will be crucial in Section~\ref{konig}.
 \end{example}

\section{Heredity chains}
\label{3a} 

 We consider an application of minors to global dimensions and semi-orthogonal decompositions. Let $(X,\kB)$ be
 a \ncs, $\kM$ be a $\kB$-module. We call $\sup\setsuch{i}{\Ext^i_\kB(\kM,\_\,)\ne0}$ the \emph{local projective
 dimension} of the $\kB$-module $\kM$ and denote it by $\lpdim_\kB\kM$. If $(X,\kB)$ is noetherian and $\kM$ 
 is coherent, then $\lpdim_\kB\kM=\sup\setsuch{\pdim_{\kB_x}\kM_x}{x\in X}$.

\begin{lemma}\label{gdim} 
 Let $(X,\kB)$ be a \ncs, $\kP$ be a \flop\ $\kB$-module, $\kA=(\END_\kB\kP)\op$ and $\okB=\kB/\kI_\kP$. Suppose that
 $\kP$ is flat as right $\kA$-module,
 \begin{align*}
  \lpdim_\kB\kI_\kP&= d,\\
  \gdim \kA&= n,\\
  \gdim \okB&= m.
 \end{align*}
 Then $\gdim\kB\le\max\set{m+d+2,\,n}$.
 \end{lemma}
 \begin{proof}
  Let $\okB=\kB/\kI_\kP$. Then $\lpdim_\kB\okB=d+1$, and from the spectral sequence 
  $\ext^p_\okB(\kM,\Ext^q_\kB(\okB,\_\,))\Arr\ext^{p+q}_\kB(\kM,\_\,)$ it follows that 
  $\pdim_\kB\kM\le m+d+1$ for every $\okB$-module $\kM$. 
 Consider the functors $\sG=\hom_\kB(\kP,\_\,)$ and $\sF=\kP\*_\kA\!\_$. Since the morphism $\sG\sF\sG\to\sG$,
 arising from the adjunction, is an isomorphism, the kernel and the cokernel of the natural map
 $\al:\sF\sG \kM\to \kM$ are annihilated by $\sG$, so are actually $\okB$-modules. It implies that 
 $\ext^i_\kB(\kM,\kN)\simeq\ext^i_\kB(\sF\sG \kM,\kN)$ if $i>m+d+2$, so 
 $\pdim_\kB \kM\le\max\set{m+d+2,\,\pdim_\kB\sF\sG \kM}$. As both functors $\sF$ and $\sG$ are exact,
 $\ext^i_\kB(\sF\_\,,\_\,)\simeq\ext^i_\kA(\_\,,\sG\_\,)$, so $\pdim_\kB\sF\sG\kM\le n$. 
 \end{proof} 
 
 \begin{definition}\label{relate} 
 \begin{enumerate}
 \item  Let $(X,\kB)$ and $(X,\kA)$ be two \ncss. A \emph{relating chain} between $\kB$ and $\kA$ is a 
 sequence $(\kB_1,\kP_1$, $\kB_2,\kP_2$,$\dots,\kP_r,\kB_{r+1})$, where $\kB_1=\kB$, $\kB_{r+1}=\kA$, 
 every $\kP_i$ ($1\le i\le r$) is a \flop\ $\kB_i$-module which is also flat as right $\kA_i$-module, where 
 $\kA_i=(\END_{\kB_i}\kP_i)\op$, and $\kB_{i+1}=\kB_i/\kI_{\kP_i}$ for $1\le i\le r$. 
 
 \item  The relating chain is said to be \emph{flat} if, for every $1\le i\le r$, $\kI_{\kP_i}$ is flat as right
 $\kB_i$-module. Note that it is the case if the natural map $\kP_i\*_{\kA_i}\kP_i\ch\to\kB_i$ is a monomorphism.
 
 \item The relating chain is said to be \emph{pre-heredity} if, for every $1\le i\le r$, $\kI_{\kP_i}$ is \lop\ as
 left $\kB_i$-module. If it is both pre-heredity and flat, it is said to be \emph{heredity}.
 
 \item If the relating chain is heredity and all \ncss\ $\kA_i$ are hereditary, i.e. $\gdim\kA_i\le1$, we say that
 the \ncs\ $\kB$ is \emph{\qhr} of level $r$. (Thus \qhr\ of level $0$ means hereditary).
 \end{enumerate}
 \end{definition}

 We fix a relating chain $(\kB_1,\kP_1$, $\kB_2,\kP_2$,$\dots,\kP_r,\kB_{r+1})$ between $\kB$ and $\kA$
 and keep the notations of Definition~\ref{relate}\,\textit{1}. 
 
 \begin{corollary}\label{bound}
 Let $\gdim\kA_i\le n$ and $\lpdim_{\kB_i}\kI_{\kP_i}\le d$ for all $1\le i\le r$. Then
 $\gdim\kB\le r(d+2)+\max\set{\gdim\kA,\,n-d-2}$.  
 If this relating chain is pre-heredity, then $\gdim\kB\le\gdim\kA+2r$.
 \end{corollary} 
 
 Using Theorem~\ref{DFG}\,(4), Theorem~\ref{qmod} and induction, we also get the following result.
 
 \begin{corollary}\label{sorth} 
  If this relating chain is flat, there are semi-orthogonal decompositions
  $(\cT,\cT_r,\dots,\cT_1)$ and $(\lst{\cT'}{r},\cT)$ of $\cD\kB$ such that 
  $\cT_i\simeq\cT'_i\simeq\cD\kA_i$ $\,(1\le i\le r)$ and $\cT\simeq\cD\kA$. 
 \end{corollary}
 
 Note that, as a rule, $\kT_i\ne\kT_i'$.
 
  \begin{corollary}\label{qhr} 
  If a \ncs\ $\kB$ is \qhr\ of level $r$, then $\gdim\kB\le2r+1$ and there are semi-orthogonal decompositions
  $(\cT,\cT_r,\dots,\cT_1)$ and $(\lst{\cT'}{r},\cT)$ of $\cD\kB$ such that $\cT_i\simeq\cT'_i$ $\,(1\le i\le r)$,
  as well as $\cT$, is equivalent to the derived category of a hereditary \ncs.
  \end{corollary} 
 
  \begin{remark}\label{qhr-rem} 
 Suppose that $(X,\kB)$ is affine: $X=\spec\qR$ and $\kB=\qB^\sim$.
\begin{enumerate}
\item If $\qB$ is semiprimary, then $\kB$ is \qhr\ with respect to our definition \iff $\qB$ is \qhr\ in the
classical sense of \citep{cps,dr}.

\item If $\qR$ is a \dvr\ and $\qB$ is an $\qR$-order in a separable algebra, then $\kB$ is \qhr\ with respect 
to our definition \iff $\qB$ is \qhr\ in the sense of \citep{ko1}.
\end{enumerate}
 \end{remark}

 \begin{example}\label{qhr-ex} 
 Consider the endomorphism construction of Example~\ref{endo}. Suppose that $\kF$ is flat as right $\kE$-module, 
 $\kF'$ is \lop\ as left $\kE$-module and the natural map $\mu_\kF:\kF\*_\kE\kF'\to\kA$ is a monomorphism. Let 
 $\tkP=\smtr{\kF\\\kE}$ and $\okA=\kA/\im\mu_\kF$. Then one can easily verify that $(\kA_\kF,\tkP,\okA)$ 
 is a heredity relating chain. Therefore, if both $\kE$ and $\okA$ are \qhr, so is $\kA_\kF$. These conditions hold,
 for instance, if $\kA$ is noetherian and reduced, $\kF$ is coherent torsion free and  $\kE$ is hereditary (the situation
 which will be explored in Section~\ref{konig}).
 \end{example}

\section{Strongly Gorenstein schemes}
\label{s4} 

 In this section we only consider \emph{noetherian} \ncss.

 \begin{definition}\label{41} 
 Let $(X,\kA)$ be a noetherian \ncs. We call it \emph{\gor} if $X$ is equidimensional, $\kA$ is \CM\
 as $\kO_X$-module and $\idim_\kA\kA=\dim X$.\!%
 \footnote{\,We do not know whether the last condition implies the \CM\ property, as it is in the
 commutative case.}
 \end{definition}
 
 Recall that an $\kA$-module $\kM$ is injective \iff $\kA_x$-modules $\kM_x$ are injective for all
 $x\in X\cl$ (the proof from \cite[Proposition~7.17]{ha} remains valid in non-commutative situation too).
 We need some basic facts about injective dimension for non-commutative rings. Now $\qR$ denotes a
 noetherian commutative local ring with the maximal ideal $\gM$ and the residue field $\aK=\qR/\gM$, 
 $\qA$ denotes an $\qR$-algebra finitely generated as $\qR$-module. Let also $\gR=\rad\qA$ and 
 $\oqA=\qA/\gR$. As usually, for every ideal $I\sbe\qR$ we denote by $V(I)$ the set of prime ideals 
 containing $I$.
 
 \begin{theorem}\label{42} 
  $\idim M=\sup\setsuch{i}{\ext^i_{\qA}(\oqA,M)\ne0}$.
 \end{theorem}
 
 Just as in \cite[Proposition~3.1.14]{bh}, this theorem is an immediate consequence of the following lemma.
   
 \begin{lemma}\label{43} 
 Let $\gP\ne\gM$ be a prime ideal of $\qR$, $M$ be a noetherian $\qR$-module. Suppose that 
 $\ext^i_{\qA}(N,M)=0$ for any noetherian $\qA$-module $N$ such that $V(\ann_{\qR} N)\sb V(\gP)$ 
 and $i>m$. Then also $\ext^i_{\qA}(N,M)=0$ for any noetherian $\qA$-module $N$ such that 
 $V(\ann_{\qR} N)= V(\gP)$ and $i>m$.
 \end{lemma}
 \begin{proof}
  Suppose that the condition is satisfied and let $V(\ann_{\qR}N)= V(\gP)$. If $\gQ\in\ass N$ and 
  $\gQ\ne \gP$, there is a submodule $N'\sbe N$ such that $\gQ N'=0$. Therefore, $\ext^i_{\qA}(N',M)=0$
  for $i>m$ and we only have to prove that $\ext^i_{\qA}(N/N',M)=0$ for $i>m$. Thus we can suppose that
  $\ass N=\{\gP\}$. Let $a\in\gM\=\gP$. Then $a$ is non-zero-divisor on $N$, i.e. we have the exact 
  sequence $0\to N\xarr a N\to N/aN\to0$. It gives and exact sequence
  \[
   \ext^i_{\qA}(N,M)\xarr a \ext^i_{\qA}(N,M) \to \ext^{i+1}_{\qA}(N/aN,M).
  \]
  Obviously, $\ann_{\qR}N/aN\sp \gP$, so the last term is $0$ if $i>m$. Therefore, 
  $a\ext^i_{\qA}(N,M)=\ext^i_{\qA}(N,M)$ and $\ext^i_{\qA}(N,M)=0$ by Nakayama's Lemma.
 \end{proof}
 
 \begin{corollary}\label{44} 
  Let $\kM$ be a \cht\ $\kA$-module. Then
\begin{align*}
    \idim_\kA\kM&=\sup\setsuch{i}{\ext^i_\kA(\kA(x),\kM)\ne0 \text{\emph{ for some }} x\in X\cl}=\\
    			&=\sup\setsuch{\idim_{\kA_x}\kM_x}{x\in X\cl}.
\end{align*}
  Here $\kA(x)$ denotes $\kA\*_{\kO_X}\aK(x)$.
 \end{corollary}
 
 \begin{corollary}\label{44a} 
 \begin{align*}
  \gdim\kA&=\sup\setsuch{\pdim_\kA\kA(x)}{x\in X\cl}=\\
  		  &=\sup\setsuch{i}{\ext^i_\kA(\kA(x),\kA(x))\ne0 \text{\emph{ for some }}x\in X\cl}=\\
  		  &=\sup\setsuch{\gdim\kA_x}{x\in X\cl}.
 \end{align*}
 \end{corollary}
 
 \begin{lemma}\label{45} 
  Let $M$ be a noetherian $\qA$-module. If an element $a\in\qR$ is non-zero-divisor both on $\qA$
  and on $M$, then $\idim_{\qA}M=\idim_{\qA/a\qA}M/aM$.
 \end{lemma}
 \begin{proof}
  It just repeats that of \cite[Corollary~3.1.15]{bh}. 
 \end{proof}
 
 \begin{corollary}\label{46} 
  Let $\fA=\row am$ be an $\qA$-sequence in $\gM$. Then $\qA$ is \gor\ \iff so is $\qA/\fA\qA$.
 \end{corollary}
 
 \begin{corollary}\label{47} 
  $\kA$ is \gor\ \iff so is $\kA\op$.
 \end{corollary}
 \begin{proof}
  The claim is local, so we can replace $\kA$ by $\qA$.
  Corollary~\ref{46} reduces the proof to the case when $\kdim\qR=0$, i.e. $\qA$ is just an artinian
  algebra. Then it is well-known \cite[Proposition~IV.3.1]{ars}. 
 \end{proof}
 
 For a noetherian \ncs\ $(X,\kA)$ we denote by $\cm\kA$ the full subcategory of $\kA\md$ consisting
 of such modules $\kM$ that $\kM_x$ is a maximal \CM\ module over $\kO_{X,x}$ for every point $x\in X$. 
 The following results can be proved just as in the commutative case (see \cite[Section~3.3]{bh}).
 
 \begin{theorem}\label{48} 
  Let $(X,\kA)$ be a \gor\ \ncs\ and $\kM\in\cm\kA$. 
  \begin{enumerate}
  \item  $\Ext^i_\kA(\kM,\kA)=0$ for $i\ne0$.
  
  \item The natural map $\kM\to\Hom_\kA(\Hom_\kA(\kM,\kA),\kA)$ is an isomorphism.
  \end{enumerate}
  \smallskip
  Thus the functor $^*:\kM\mps\kM^*=\Hom_\kA(\kM,\kA)$ gives an exact duality between the categories
  $\cm\kA$ and $\cm\kA\op$.
 \end{theorem}

 Let now $(X,\kA)$ be a \gor\ \ncs, $\kF\in\cm\kA$. Consider the endomorphism construction described
 in Example~\ref{endo}. Theorem~\ref{48} implies that the natural map $\phi(\kM):\sF_\kF\kM\to\sH_\kF\kM$ 
 is an isomorphism for $\kM=\kA$, hence an isomorphism for any $\kM\in\lp\kA$.
  
 \begin{theorem}\label{49} 
 Let $(X,\kA)$ be \gor\ and contain enough \lop\ modules, $\kF\in\cm\kA$. Then the restrictions of the  
 functors $\LF_\kF$ and $\RH_\kF$ onto the subcategory $\cD\sm\kA$ are isomorphic. Thus the restriction
 of $\LF_\kF$ onto $\cD\sm\kA$ is both left and right adjoint to the bilocalization functor 
 $\DG_\kF$.
 \end{theorem}
 \begin{proof}
   As $\kA$ has enough \lop\ modules, any complex from $\cD\sm\kA$ is quasi-isomorphic to a finite complex
   $\kC\sd$ such that all $\kC^i$ are from $\lp\kA$. Then $\LF_\kF\kC\sd=\sF_\kF\kC\sd$. On the other hand,
   by Theorem~\ref{48}, $\sR^k\sH_\kF\kC^i=\Ext^k_\kA(\kP_\kF,\kC^i)=0$ for $k\ne0$. Therefore,
   $\RH_\kF\kC\sd=\sH_\kF\kC\sd\simeq\sF_\kF\kC\sd$. 
 \end{proof}

\section{Non-commutative curves}
\label{s5}

\subsection{Generalities}
\label{s51} 

 \begin{definition}\label{51} 
  A \emph{\ncc} is a reduced \ncs\ $(X,\kA)$ such that $X$ is an excellent curve (equidimensional reduced 
  noetherian scheme of dimension $1$) and $\kA$ is coherent and torsion free as $\kO_X$-module. 
  \end{definition} 

 As $X$ is excellent, then $\hat\kA_x$, the $\gM_x$-adic completion of $\kA_x$, is also reduced (has
 no nilpotent ideals). Therefore, for the local study of \nccs\ we can use the usual results from the books
 \citep{cr,rei}. We denote by $\kK=\kK(X)$ the sheaf of full rings of fractions of $\kO_X$ and write 
 $\kK\kM$ instead of $\kK\*_{\kO_X}\kM$ for any $\kO_X$-module $\kM$. In particular, $\kK\kA$ is a $\kK$-algebra. 
 The sheaves $\kK\kM$ are locally constant; the stalks of $\kK$ and $\kK\kA$ are semi-simple rings.  
 The \emph{torsion part} $\trs\kM$ of $\kM$ is defined as the kernel of the natural map $\kM\to\kK\kM$.
 We say that a \cht\ $\kA$-module $\kM$ is \emph{torsion free} if $\trs\kM=0$, and we say that $\kM$ is 
 \emph{torsion} if $\kK\kM=0$. Note that $\trs\kM$ is torsion and $\kM/\trs\kM$ is torsion free.
 We denote by $\tors\kA$ and $\tf\kA$ respectively the full subcategories of $\kA\md$ consisting of 
 torsion and of torsion free modules. 
 We always consider a torsion free module $\kM$ as a submodule of $\kK\kM$.
 In particular, we identify $\kM_x$ with its natural image in $\kK\kM_x$. 
 Note that for every submodule $\kN\sbe\kK\kM$ there is a natural embedding
 $\kK\kN\to\kK\kM$ and we identify $\kK\kN$ with the image of this embedding. 
  A \ncc\ $(X,\kA')$ is said to be an \emph{over-ring} of a \ncc\ $(X,\kA)$ if $\kA\sbe\kA'\sb\kK\kA$.
  Then $\kA'$ is naturally considered as a coherent $\kA$-module. The \ncc\ $(X,\kA)$ is said to be
  \emph{normal} if it has no proper over-rings. Since $X$ is excellent and $\kA$ is reduced, the set
  $\setsuch{x\in X}{\kA_x \text{ is not normal}}$ is finite. Then it follows from \citep{dm} that the set 
  of over-rings of $\kA$ satisfies the maximality condition: there are no infinite strictly ascending 
  chains of over-rings of $\kA$.
 
 Coherent torsion free $\kA$-modules, in particular, over-rings of $\kA$ can be constructed locally.
 
 \begin{lemma}\label{52} 
 Let $\kM$ be a torsion free coherent $\kA$-module.
 \begin{enumerate}
 \item If $\kN$ is a coherent $\kA$-submodule of $\kK\kM$ such that $\kK\kN=\kK\kM$, then $\kN_x=\kM_x$ for 
 almost all $x\in X$.
 
 \item Let $S\sb X\cl$ be a finite set and for every $x\in S$ a finitely generated $\kA_x$-submodule 
 $N_x\sb\kK\kM_x$ is given such that $\kK N_x=\kK\kM_x$. Then there is a unique $\kA$-submodule 
 $\kN\sb\kK\kM$ such that $\kN_x=N_x$ for every $x\in S$ and $\kN_x=\kM_x$ for all $x\notin S$. 
 
 \item If $\kM=\kA$ and all $N_x$ in the preceding item are rings, then $\kN$ is a subalgebra of $\kK\kA$,
 so $(X,\kN)$ is also a \ncc\ and if $N_x\spe\kA_x$ for all $x\in S$, $(X,\kN)$ is an over-ring of $(X,\kA)$.
 \end{enumerate}
 \end{lemma}
 \begin{proof}
  We can suppose that $X$ is affine. Then the proof just repeats that of 
  \cite[Ch.\,VII,\,\S\,3,\,Theorem 3]{bou} with slight and obvious changes.
 \end{proof}
 
 \begin{lemma}\label{53} 
 Any \ncc\ $(X,\kA)$ has enough invertible modules. Namely, the set 
 \[
  \bL_\kA=\setsuch{\kA\*_{\kO_X}\kL}{\kL \text{\emph{ is an invertible $\kO_X$-module}}}
 \]
 generates $\qoh\kA$ \lb hence, generates $\cD\kA$\rb.
 \end{lemma}
 \begin{proof}
  We must show that if $\kM'\sb\kM$ is a proper submodule, there is a homomorphism $f:\kL\to\kM$ such that
  $\im f\not\sbe\kM'$. As $\hom_\kA(\kA\*_{\kO_X}\kL,\kM)\simeq\hom_{\kO_X}(\kL,\kM)$, we can suppose that
  $\kA=\kO_X$. Moreover, as every $\kA$-module is a direct limit of its coherent submodules, we can suppose
  that $\kM$ is coherent. Let first $\kM'\not\spe\trs\kM$. Choose $x\in X\cl$ such that 
  $\trs\kM_x\not\sbe\kM'_x$ and let $u_x\in\trs\kM_x\=\kM'_x$. There is a global section 
  $u\in\Ga(X,\trs\kM)\sbe\Ga(X,\kM)$ such that $u_x$ is its image in $\kM_x$. Then there is a homomorphism 
  $f:\kO_X\to\kM$ such that $f\textit{1}=u$, so $\im f\not\sbe\kM'$. 
  
  Let now $\kM'\spe\trs\kM$. Since $\ext^1_{\kO_X}(\kL,\trs\kM)=0$ for any \lop\ module $\kL$, 
  the map $\hom_{\kO_X}(\kL,\kM)\to\hom_{\kO_X}(\kL,\kM/\trs\kM)$ is surjective. Hence, we can suppose that
  $\kM$ is torsion free. Let $\kM_y\ne\kM'_y$ for some $y\in X\cl$ and $u_y\in\kM_y\=\kM'_y$. There is a
  homomorphism $\vi:\kK\to\kK\kM$ such that $\vi\textit{1}=u_y$. Let $\kN=\vi(\kO_X)$. The set 
  $S=\setsuch{x\in X\cl}{\kN_x\not\sbe\kM_x}$ is finite; moreover, $y\notin S$. For every $x\in S$ there is
  an ideal $L_x\sbe\kO_{X,x}$ such that $L_x\simeq\kO_{X,x}$ and $\vi(L_x)\sbe\kM_x$. By Lemma~\ref{52} there is 
  an ideal $\kL\sbe\kO_X$ such that $\kL_x=L_x$ for $x\in S$ and $\kL_x=\kO_{X,x}$ otherwise. It is an
  invertible ideal, $\vi(\kL)\sbe\kM$ and $\vi(\kL)\not\sbe\kM'$.
 \end{proof}
 
  We will use the duality for left and right \cht\ torsion free $\kA$-modules established 
 in the following theorem.
 
\begin{theorem}\label{canon} 
\begin{enumerate}
\item   There is a \emph{canonical $\kA$-module}, i.e. such a module $\om_\kA\in\tf\kA$ that 
 $\idim_{\kA}\om_\kA=1$ and $\END_\kA\om_\kA\simeq\kA\op$ \lb so $\om_\kA$ can be considered as an 
 \emph{$\kA$-bimodule}\rb. Moreover, $\om_\kA$ is isomorphic as a bimodule to an ideal of $\kA$.
 
 \smallskip\noindent
 \emph{We denote by $\kM^*$, where $\kM\in\kA\Md$ (or $\kM\in\kA\op\Md$) the $\kA\op$-module
 (respectively, $\kA$-module) $\Hom_\kA(\kM,\om_\kA)$ (respectively, $\Hom_{\kA\op}(\kM,\om_\kA)$).}
 
 \smallskip
\item The natural map $\kM\to\kM^{**}$ is an isomorphism for every $\kM\in\tf\kA$ \lb or
 $\kM\in\tf\kA\op$\rb\ and the functors $\kM\mps\kM^*$ establish an exact duality of the 
 categories $\tf\kA$ and $\tf\kA\op$. Moreover, if $\kM\in\kA\md$, then 
 $\kM^{**}\simeq\kM/\trs\kM$.
\end{enumerate}
\end{theorem}
\begin{proof}
 Each local ring $\kO_x=\kO_{X,x}$ is excellent, so its integral closure in $\kK_x$ is finitely
 generated and its completion $\hat\kO_x$ is reduced. Therefore $\kO_x$ has a canonical module 
 $\om_x$ which can be considered as an ideal in $\kO_x$ \cite[Korollar~2.12]{hk}. Moreover, $\kO_x$ 
 is normal for almost all $x\in X\cl$ and in this case we can take $\om_x=\kO_{X,x}$. By Lemma~\ref{52}, 
 there is an ideal $\om_X\sbe\kO_X$ such that $\om_{X,x}=\om_x$ for each $x\in X$. Then 
 $\idim_{\kO_X}\om_X=\sup\set{\idim_{\kO_{X,x}}\om_x}=1$. As the natural map
 $\kO_{X,x}\to\End_{\kO_{X,x}}\om_x$ is an isomorphism for each $x\in X$,
 the natural map $\kO_X\to\END_{\kO_X}\om_X$ is an isomorphism too.
 Therefore, $\om_X$ is a canonical $\kO_X$-module. Then it is known that 
 the functor $\kM\mps\kM^*=\Hom_{\kO_X}(\kM,\om_X)$
 is an exact self-duality of $\tf\kO_X$ and the natural map $\kM\to\kM^{**}$ is an isomorphism.
 Set now $\om_\kA=\Hom_{\kO_X}(\kA,\om_X)$. Then $\Hom_\kA(\kM,\om_\kA)\simeq\Hom_{\kO_X}(\kM,\om_X)$
 for any $\kA$-module $\kM$, whence all statements of the theorem follow.
\end{proof}

  As usually, we say that two \ncss\ $(X,\kA)$ and $(Y,\kB)$ are \emph{Morita equivalent} if their categories
 of \qch\ modules are equivalent. A coherent \lop\ $\kA$-module $\kP$ is said to be a \emph{local progenerator} 
 if $\kP_x$ is a progenerator for $\kA_x$ for all $x\in X$, that is $\kP_x$ is projective over $\kA_x$ and there is a surjection
 $r\kP_X\to\kA_x$ for some $r$. It follows from Theorem~\ref{FG} that then $(X,\kA)$ 
 is Morita equivalent to $(X,\kE)$, where $\kE=(\END_\kA\kP)\op$.
 
 \begin{theorem}\label{55} 
 \begin{enumerate}
 \item   Let $(X,\kA)$ and $(X,\kB)$ are two \nccs\ such that $\kA_x$ is Morita equivalent to $\kB_x$ for every
  $x\in X\cl$. Then $(X,\kA)$ and $(X,\kB)$ are Morita equivalent.
  
  \item  Let now $(X,\kA)$ and $(Y,\kB)$ be two central \nccs\ finite over a field. If they are Morita equivalent,
  there is an isomorpism $\tau:X\ito Y$ such that, for every points $x\in X$ and $y=\tau(x)$, the rings 
  $(\tau^*\kB)_x$ and $\kA_x$ are Morita equivalent.
 \end{enumerate}

 \end{theorem}
 \begin{proof}
  \textit{1} If $\kA_x$ and $\kB_x$ are Morita equivalent, there is a progenerator $P_x$ for $\kA_x$ such that
  $\kB_x\simeq(\End_{\kA_x}P_x)\op$. There is a $\kK\kA$-module $\kV$ such that $\kV\simeq\kK P_x$ for
  all $x\in X\cl$. Choose a normal over-ring $\kA'$ of $\kA$ and a coherent $\kA'$-submodule $\kM\sb\kV$
  such that $\kK\kM=\kV$. Then $\kM$ is a local progenerator for $\kA'$. Set $\kB'=(\END_{\kA'}\kM)\op$
  and $S=\setsuch{x\in X\cl}{\kA_x\ne\kA'_x \text{ or } \kB_x\ne\kB_x'}$. The set $S$ is finite, so there is an
  $\kA$-submodule $\kP\sb\kV$ such that $\kP_x=P_x$ for $x\in S$ and $\kP_x=\kM_x$ for $x\notin S$.
  Then $\kP$ is a local progenerator for $\kA$ and $\kB\simeq(\END_\kA\kP)\op$.
  
  \smallskip
  \textit{2} follows from \cite[Section~6]{az}.
 \end{proof}
 
\subsection{Hereditary \nccs}
\label{s52}   
 
 We call a noetherian \ncs\ $(X,\kA)$ \emph{hereditary} if all localizations $\kA_x$ are hereditary rings, 
 i.e. $\gdim\kA_x=1$. Then $\gdim\kA=1$ too, so $\ext^2_\kA(\kM,\kN)=0$ for all $\kA$-modules $\kM,\kN$. 
 Suppose that $(X,\kA)$ is a hereditary \ncc. Then any torsion free coherent $\kA$-module $\kM$ is \lop, 
 so $\Ext^1_\kA(\kM,\kN)=0$ for any $\kA$-module $\kN$. If $\kN$ is \cht\ and torsion, it implies that 
 $\ext^1_\kA(\kM,\kN)=0$. Therefore, every \cht\ $\kA$-modules $\kM$ splits as
 $\kM=\trs\kM\+\kM'$, where $\kM'$ is torsion free, hence \lop. If a central \ncc\ $(X,\kH)$ is hereditary, 
 then $X$ is smooth. There is an effective description of hereditary \nccs\ up to Morita equivalence. 
 
 First consider the case when $X=\spec\qO$, where $\qO$ is a complete \dvr\ with the field of fractions
 $\qK$, the maximal ideal $\gM$ and the residue field $\aK=\qO/\gM$. Let $\qH$ be a hereditary reduced
 $\qO$-algebra which is torsion free as $\qO$-module. Then $\qK\qH\simeq\Mat(n,\qD)$, where $\qD$ is a finite 
 dimensional division algebra over $\qK$. There is a unique maximal $\qO$-order $\De\sb\qD$ \cite[Theorem~12.8]{rei}.
 It contains a unique maximal ideal $\dM$, which is both left and right principal ideal. Let $n=\sum_{i=1}^kn_i$
 for some positive integers $n_i$, $\fN=\row nk$ and $\qH(\fN,\qD)$ be the subring of $\Mat(n,\De)$ consisting
 of $k\xx k$ block matrices $(A_{ij})$ such that $A_{ij}$ is of size $n_i\xx n_j$ and if $j>i$ all coefficients
 of $A_{ij}$ are from $\dM$. Let also $L=\De^n$ considered as $\qH(\fN,\qD)$-module and $L_i$ be the submodule 
 in $L$ consisting of such vectors $\row\al n$ that $\al_j\in\dM$ for $j\le\sum_{q=1}^in_q$. In particular, 
 $L_0=L$ and $L_k=\dM^n\simeq L$. If necessary, we denote $L_i=L_i(\qH)$.
 
 \begin{theorem}{\rm\cite[Theorem~39.14]{rei}}\label{56} \
 Let $\qO$ be a complete \dvr.
 \begin{enumerate}
 \item  Every connected hereditary $\qO$-order is isomorphic to $\qH(\fN,\qD)$ for some tuple
 $\fN=\row nk$, which is uniquely determined up to a cyclic permutation.
 
 \item  Hereditary orders $\qH(\fN,\qD)$ and $\qH(\fN',\qD')$ are \meq\ \iff $\qD\simeq\qD'$ and $\fN$ and $\fN'$
 are of the same length.
 
 \item $L_i\ (0\le i<k)$ are all indecomposable projective $\qH(\fN,\qD)$-modules and $U_i=L_i/L_{i+1}$ are all simple 
 $\qH(\fN,\qD)$-modules \lb up to isomorphism\rb.
 \end{enumerate}
 \end{theorem}
  
 Let now $(X,\kH)$ be a connected central hereditary \ncc. Then $\kK\kH$ is a central simple $\kK$-algebra:
 $\kK\kH=\Mat(n,\kD)$, where $\kD$ is a central division algebra. For every closed point $x\in X$ the completion
 $\hat\kD_x$ is isomorphic to $\Mat(m_x,\qD_x)$ for some central division algebra $\qD_x$ over $\hat\kK_x$
 and some integer $m_x=m_x(\kD)$.
 Therefore, for every closed point $x\in X$, $\hat\kH_x$ is isomorphic to $\qH(\fN,\qD_x)$ for some
 $\fN=\row nk$, where $\sum_{i=1}^kn_i=m_xn$. Thus Theorems~\ref{55} and \ref{56} give the following result.

 \begin{theorem}\label{57} 
  A central hereditary \ncc\ $(X,\kH)$ is determined up to \meqc\ by a central division $\kK$-algebra 
 $\kD$ and a function $\ka:X\cl\to\mN$ such that $\ka(x)=1$ for almost all $x\in X\cl$.
 
 \medskip
 \emph{Actually, $\ka(x)$ is the number of non-isomorphic simple $\kH$-modules $\kU$ such that  $\supp\kU=\{x\}$.}
\end{theorem}   
 
 \begin{remark}\label{58} 
 Representatives of a class given by $\kD$ and $\ka$ can be obtained as follows. Choose an integer $n$ such that
 $\ka(x)\le nm_x(\kD)$ for all $x\in X\cl$. There is an $\kO_x$-order $\qH_x$ in $\Mat(n,\kD)$ such that
 $\hqH_x=\qH(\fN_x,\qD_x)$ for some $\fN_x=(n_{1,x},n_{2,x},\dots,n_{\ka(x),x})$. Fix a normal \ncc\ $(X,\bde)$ such
 that $\kK\bde=\kD$. Then we can define $\kH=\kH(\fN,\kD)$ as the \ncc\ such that $\kK\kH=\Mat(n,\kD)$, 
 $\kH_x=\Mat(n,\bde_x)$ if $\ka(x)=1$ and $\kH_x=\qH_x$ if $\ka(x)>1$.   
 \end{remark}

 Let $S=\setsuch{x\in X}{\ka(x)>1}$, $\kL=\bde^n$ considered as $\kH$-module. Consider the submodules $\kL_{x,i}$ 
 ($0\le i\le\ka(x)$) such that $(\widehat{\kL_{x,i}})_x=L_i(\hqH_x)$ and $(\kL_{x,i})_y=\kL_y$ if $y\ne x$.
 Let also $U_{x,i}=\kL_{x,i}/\kL_{x,i+1}$ ($0\le i<\ka(x)$). Then $U_{x,i}$ are all simple $\kH$-modules (up to
 isomorphism). Note that $\kL_{x,0}=\kL$ for every point $x$.
 
 \begin{theorem}\label{59} 
 Let $\kH=\kH(\fN,\kD)$.
 \begin{enumerate}
 \item   The set 
  $$\mL_\kH=\{\kL\}\cup\setsuch{\kL_{x,i}}{x\in S,\,1\le k\le\ka(x)}$$
  classically generates $\cD\sm\kH$, hence generates $\cD\kH$ \lb see \cite[Theorem~2.2]{lun}\rb.
  
  \item  $\cD\kH\simeq\cD\mA$, where $\mA$ denotes the DG-category with the set of objects $\mL_\kA$  
  and $\mA(\kL',\kL'')=\rhom_\kA(\kL',\kL'')$.
 \end{enumerate} 
 \end{theorem}
 \begin{proof}
  \textit{1.} Obviously, $\gnr{\mL_\kH}_\8$ contains all simple $\kH$-modules. Therefore, it contains all
  torsion \cht\ $\kH$-modules, as well as all \cht\ $\kH$-submodules of $\kK\kL$. If $\kM$ is a \cht\ \TF\ 
  $\kH$-module, it contains a submodule $\kN$ isomorphic to a submodule of $\kK\kL$ such that $\kM/\kN$
  is also \TF. It implies that $\gnr{\mL_\kH}_\8$ contains all \cht\ $\kH$-modules, hence coincides with
  $\cD\sm\kH$.
  
  \textit{2}\, follows now from \cite[Proposition~2.6]{lun}.
 \end{proof}
 
  \begin{corollary}\label{510} 
  Let $\aK$ be an algebraically closed field. 
  \begin{enumerate}
  \item A connected hereditary algebraic \ncc\ over $\aK$ is defined up to \meqc\ by a pair $(X,\ka)$, where $X$ is a
  smooth connected algebraic curve over $\aK$ and $\ka:X\cl\to\mN$ is a function such that $\ka(x)=1$
  for almost all $x$. Representatives of the Morita class given by such a pair are $\kH(\fN,\kK)$ as described 
  in Remark~\ref{58}.
  
  \item Two connected hereditary \nccs\ given by the pairs $(X,\ka)$ and $(X',\ka')$ are \meq\ \iff there is
  an isomorphism $\tau:X\to X'$ such that $\ka'(\tau(x))=\ka(x)$ for all $x\in X\cl$.
  \end{enumerate}
  
  \emph{In this case we write $\kH(\fN,X)$ instead of $\kH(\fN,\kK)$.}
  \end{corollary}   
  \begin{proof}
   The Brauer group of $\kK$ is trivial \cite[Theorem~17]{lang}. Therefore, the algebra $\kD$ in 
   Theorem~\ref{57} coincides with $\kK$.
  \end{proof}

 We say that a central \ncc\ $(X,\kA)$ is \emph{rational} (over a field $\aK$) if all simple components of the algebra 
 $\kK\kA$ are of the form $\Mat(n,\kK)$. Then the curve $X$ is also rational over $\aK$.
 
 \begin{theorem}\label{511} 
 Let $(X,\kH)$ be a connected rational hereditary \ncc\ over a field $\aK$ and $\ka:X\cl\to\mN$ 
 be the corresponding function. Let $S=\setsuch{x\in X\cl}{\ka(x)>1}$, $o\in X\cl$ be an arbitrary point.
 \begin{enumerate}
 \item  The set
\[ 
 \omL_\kH=\set{\kL,\kL(-o)}\cup\setsuch{\kL_{x,i}}{x\in S,\,1\le i<\ka(x)}
\]
 classically generates $\cD\sm\kH$, hence generates $\cD\kH$.
 
 \item If $\kL',\kL''\in\omL_\kH$, then $\ext^k_{\kH}(\kL',\kL'')=0$ for all $k>0$, while
 \[
  \dim\hom_{\kH}(\kL',\kL'')=\begin{cases}
   1 & \text{ \emph{if} } \kL'=\kL'', \\
   	 & \text{ \emph{or} } \kL'=\kL(-o),\,\kL''=\kL_{x,i},\\
   	 & \text{ \emph{or} } \kL'=\kL_{x,i},\,\kL''=\kL,\\
   	 & \text{ \emph{or} } \kL'=\kL_{x,j},\,\kL''=\kL_{x,i},\,j>i,\\
   2 & \text{ \emph{if} } \kL'=\kL(-o),\,\kL''=\kL,\\ 
   0 &  \text{ \emph{in all other cases.} }  	
   \end{cases}
 \]
  In particular, $\omL_\kH$ is a tilting set for the category $\cD\kH$.
 
 \item If $\th_{x,i}$ are generators of the spaces $\hom_\kH(\kL_{x,i},\kL_{x,i-1})$ $(1\le i\le \ka(x))$, 
 then the products $\th_x=\th_{x,1}\th_{x,2}\dots\th_{x,\ka(x)}$ are non-zero and any two of them generate 
 $\hom_\kH(\kL(-o),\kL)$. 
 \end{enumerate}
 \end{theorem}
 \begin{proof}
 \textit{1.} If $X\simeq\mP^1$, then all sheaves $\kO(-x)$, hence all sheaves $\kL(-x)$ are isomorphic. Moreover, in this case 
 $\kL_{x,\ka(x)}\simeq\kL(-x)$ for any $x\in X\cl$, so we can apply Theorem~\ref{59}.
 
 \textit{2.} From the definition of $\kL$ and $\kL_{x,i}$ it immediately follows that
 \[
  \Hom_{\kH}(\kL',\kL'')\simeq\begin{cases}
  \kO & \text{ {if} } \kL'=\kL'', \\
   	 & \text{ {or} } \kL'=\kL_{x,i},\,\kL''=\kL,\\
   	 & \text{ {or} } \kL'=\kL_{x,j},\,\kL''=\kL_{x,i},\,j>i,\\
  \kO(o-x) & \text{ {if} } \kL'=\kL(-o),\,\kL''=\kL_{x,i},\\
  \kO(o) & \text{ {if} } \kL'=\kL(-o),\,\kL''=\kL,\\  
  \kO(-o) &  \text{ {in all other cases.} }  	
   \end{cases}
 \]
 Since $\ext^i_\kH(\kL',\kL'')=H^i(\Hom_\kH(\kL',\kL''))$, it implies the statement.
 
 \textit{3.} One easily sees that, if $x=(1:\xi)$ as the point of $\mP^1$, then $\th_x$, 
 up to a scalar, is the multiplication by $t-\xi$, where $t$ is the affine coordinate 
 on the affine chart $\mA^1_0$. Now the statement is obvious.
 \end{proof}
 
 Recall that a \emph{canonical algebra} \cite[3.7]{ri} is given by a sequence of integers
 $\row kr$, where $r\ge2$ and all $k_i\ge2$ if $r>2$, and a sequence $(\la_3,\la_4,\dots,\la_r)$ 
 of different non-zero elements from $\aK$ (if $r=2$, this sequence is empty). Namely, this algebra, which
 we denote by $\qR(\lst kr;\la_3,\dots,\la_r)$, is given by the quiver
 \begin{equation}\label{can} 
 \xymatrix@R=1ex@C=3em{ & \bu \ar[r]^{\al_{21}} & \bu\ \dots\ \bu \ar[r]^{\al_{k_1-1,1}} &
 	\bu \ar[dr]^{\al_{k_11}} \\
 \bu \ar[ur]^{\al_{11}} \ar[r]_{\al_{12}} \ar[ddr]_{\al_{1r}} 
 & \bu \ar[r]_{\al_{22}} & \bu\ \dots\ \bu \ar[r]_{\al_{k_2-1,2}} &
 	\bu \ar[r]_{\al_{k_22}} &\bu \\
 	&& \vdots \\
  & \bu \ar[r]_{\al_{2r}} & \bu\ \dots\ \bu \ar[r]_{\al_{k_r-1,r}} &
 	\bu \ar[uur]_{\al_{k_rr}} 
 		} 
 \end{equation}
 with relations $\al_j=\al_1+\la_j\al_2$ for $3\le j\le r$, where $\al_j=\al_{k_jj}\dots\al_{2j}\al_{1j}$.
 Certainly, if $r=2$, it is the path algebra of a quiver of type $\ti\rA_{k_1+k_2}$. In particular, if
 $r=2,\,k_1=k_2=1$, it is the Kronecker algebra.
 
 \begin{corollary}\label{512} 
 Let $(X,\kH)$ be a rational projective hereditary \ncc, $\ka:X\cl\to\mN$ be the corresponding function.
 Let $\kT=\bop_{\kF\in\omL_\kH}\kF$ and $\La=(\End_\kH\kT)\op$. If $S=\set{\lst xr}$ with $r\ge2$, we set 
 $k_i=\ka(x_i)$. If $S=\set{x}$, we set $r=2$, $k_2=1$ and $k_1=\ka(x)$. If $S=\0$, we set $r=2,\ k_1=k_2=1$.
\begin{enumerate}
 \item  $\kT$ is a tilting $\kH$-module, i.e. $\ext^i_\kH(\kT,\kT)=0$ for $i\ne0$ and $\kT$ is a local progenerator for $\kH$.

 \item  $\La\simeq\qR(\lst kr;\la_3,\dots,\la_r)$ for some $\la_3,\dots,\la_r$.
  
 \item  The functor $\hom_\kH(\kT,\_\,)$ induces an equivalence $\cD\kH\simeq\cD\La$.
 \end{enumerate} 
 \end{corollary}

 Actually, the preceding considerations also show that a rational projective hereditary \ncc\ is
 Morita equivalent to a \emph{weighted projective line} \citep{gl}. It can also be
 deduced from the description of hereditary \nccs\ and the remark on page~271 of \citep{gl}.
 
 \section{K\"onig resolution and tilting}
 \label{konig}  
 
 \subsection{K\"onig resolution}
 \label{sk1}  
 
 For a \ncc\ $(X,\kA)$ we denote by $\kJ=\kJ(\kA)$ its ideal defined by the localizations as follows:
 \[
  \kJ_x=
  \begin{cases}
   \kA &\text{if $\kA$ is hereditary},\\
   \rad\kA &\text{otherwise}.
  \end{cases}
 \]
 We also denote by $\kA\sh$ the \ncc\ $\END_{\kA\op}\kJ$ (the endomorphism algebra of $\kJ$ as of \emph{right} $\kA$-module).
 It can and will be identified with an over-ring of $\kA$. The following result is proved in \cite[Theorem~39.14]{rei}.
 
 \begin{proposition}
  $\kA=\kA\sh$ \iff $\kA\,$ is hereditary.
 \end{proposition}
 
 Thus we can construct a chain of over-rings
 \[
  \kA=\kA_1\sb\kA_2\sb\kA_3\sb\dots\sb\kA_{n+1}=\kH,
 \]
 where $\kA_{i+1}=\kA_i\sh\ (1\le i\le n)$  and $\kH$ is hereditary. We call $n$ the \emph{level} of $\kA$. The \ncc\
 $\tkA=(\END_\kA\kA_\+)\op$, where $\kA_\+=\bop_{i=1}^{n+1}\kA_i$ is called the \emph{K\"onig resolution} of the \ncc\ $\kA$.
 It is identified with the algebra of matrices
 \[
  \tkA=\mtr{\kA_{11} & \kA_{12} & \kA_{13} &\dots& \kA_{1,n+1} \\
				\kA_{21} & \kA_{22} & \kA_{23} &\dots& \kA_{2,n+1} \\
				\kA_{31} & \kA_{32} & \kA_{33} &\dots& \kA_{3,n+1} \\
				\hdotsfor5\\
				\kA_{n+1,1} & \kA_{n+1,2} & \kA_{n+1,3} &\dots& \kA_{n+1,,n+1} }
 \] 
 where $\kA_{ij}=\Hom_\kA(\kA_i,\kA_j)$. Note that $\kA_{ij}=\kA_j$ if $i\le j$, while $\kA_{i+1,i}\spe\kJ(\kA_i)$. Let $e_i\ (1\le i\le n+1)$ 
 be the diagonal idempotents of $\tkA$, $\kP=\tkA e_1$ and $\tkP=\tkA e_{n+1}$. Then $(\END_{\tkA}\kP)\op\simeq\kA$ and
 $(\END_{\tkA}\tkP)\op\simeq\kH$, so both $\kA$ and $\kH$ are minors of $\tkA$ and the categories $\kA\Md$ and $\kH\Md$ 
 ($\cD\kA$ and $\cD\kH$) are bilocalizations of $\tkA\Md$ (respectively, of $\cD\tkA$) with respect to bilocalization functors
 $\sG=\Hom_{\tkA}(\kP,\_\,)$ and $\tsG=\Hom_{\tkA}(\tkP,\_\,)$.
 
 We also denote $\eps_k=\sum_{j=k+1}^{n+1}e_k$, $\kI_k=\tkA \eps_k\tkA$, $\kQ_k=\tkA/\kI_k$ and $\kP_k=\kQ_ke_k$. The next
 result justifies the term ``resolution'' in the name of $\tkA$.
 
 \begin{theorem}\label{k1} 
 $\bC=(\tkA,\tkP,\kQ_n,\kP_n,\kQ_{n-1},\kP_{n-1},\dots,\kQ_2,\kP_2,\kQ_1)$ is a heredity relating chain between $\tkA$ and 
 $\kQ_1=\kA/\kA_{21}$. Moreover, $(\END_{\kQ_i}\kP_i)\op=\kA_i/\kA_{i+1,i}$ is semi-simple, so $\tkA$ is a quasi-hereditary
 \ncc\ of level $n$ and $\gdim\tkA\le2n$.
 \end{theorem}
 \begin{proof}
  One easily verifies that $\kI_i$ is the ideal of the matrices
  \[
   \kI_i=\mtr{\kA_{i1} & \kA_{i2} &\dots& \kA_{i,i-1} &\kA_{ii} & \kA_{i+1,i} & \dots & \kA_{i,n+1}\\
				 \kA_{i1} & \kA_{i2} &\dots& \kA_{i,i-1} &\kA_{ii} & \kA_{i+1,i} & \dots & \kA_{i,n+1}\\
				 \hdotsfor8 \\
				 \kA_{i1} & \kA_{i2} &\dots& \kA_{i,i-1} &\kA_{ii} & \kA_{i+1,i} & \dots & \kA_{i,n+1}\\
				 \kA_{i+1,1} & \kA_{i+1,2} &\dots& \kA_{i+1,i-1} &\kA_{i+1,i} & \kA_{i+1,i+1} & \dots & \kA_{i+1,n+1}\\
				 \hdotsfor8\\
				  \kA_{n+1,1} & \kA_{n+1,2} &\dots& \kA_{n+1,i-1} &\kA_{n+1,i} & \kA_{n+1,i} & \dots & \kA_{n+1,n+1} }
  \]
  Therefore, $\kQ_i$ is identified with the algebra of $i\xx i$ matrices $(a_{kl})$, where $a_{kl}\in\kA_{kl}/\kA_{i+1,l}$. Thus 
  $a_{ii}\in\kA_i/\kA_{i+1,i}$ and the latter algebra is semi-simple. Hence $\END_{\kQ_i}\kP_i=e_i\kQ_ie_i=\kA_i/\kA_{i,i+1}$
  is semi-simple. On the other hand, $\kI_{\kP_i}=\kQ_ie_i\kQ_i=\kI_{i+1}/\kI_i$, whence $\kQ_{i-1}=\kQ_i/\kI_{\kP_i}$, so $\bC$
  is a relating chain. Moreover, $\kI_i$ is projective as a right $\tkA$-module, hence $\kI_i/\kI_{i+1}$ is projective as a right
  $\kQ_i$-module, so this relating chain is heredity. As $\kH=(\END_{\tkA}\tkP)\op$ is hereditary, $\tkA$ is quasi-hereditary, and
  as all $\END_{\kQ_i}\kP_i$ are semi-simple, $\gdim\tkA\le2n$ by Corollary~\ref{bound}. 
 \end{proof}
 
 It means that the functor $\DG:\cD\tkA\to\cD\kA$ defines a categorical resolution of $\cD\kA$ in the sense of \citep{lun}. If
 $\kA$ is strongly Gorenstein, Theorem~\ref{49} shows that this resolution is \emph{weakly crepant}, i.e. its left and right adjoints
 coincide on perfect complexes (small objects in $\cD\kA$).
 
 We denote by $\bar\kA_i$ the  semi-simple algebra $\kA_i/\kA_{i+1,i}\simeq\END_{\kQ_i}\kP_i$.
  
  \begin{corollary}\label{k2} 
   The derived category $\cD\tkA$ has two semi-orthogonal decompositions: $\cD\tkA=\gnr{\cT_1,\cT_2,\dots,\cT_n,\cT}$ and 
   $\cD\kR=\gnr{\cT',\cT'_n,\dots,\cT'_2,\cT'_1}$, where $\cT\simeq\cT'\simeq\cD\kH$ and $\cT_i\simeq\cT'_i\simeq\cD\bar\kA_i$.
   
   \smallskip\noindent
   \emph{Note that usually $\cT\ne\cT'$ as well as $\cT_i\ne\cT'_i$ for $i>1$, though $\cT_1=\cT'_1=\cD(\tkA/\kI_2)$ 
   naturally embedded into $\cD\tkA$.}
  \end{corollary}

 Let $\sF$ and $\tsF$ be, respecrively, the left adjoints of $\sG$ and $\tsG$; $\sH$ and $\tsH$ be, respectively, the right adjoints
 of $\sG$ and $\tsG$. Then we have the diagram of bilocalizations
 \begin{equation}\label{dia} 
   \xymatrix{ \kH\Md \ar@/^1em/[rr]^{\tsF} \ar@/_1em/[rr]_{\tsH} && \tkA\Md \ar[ll]|{\,\tsG\,} \ar[rr]|{\,\sG\,} &&
    \kA\Md \ar@/_1em/[ll]_{\sF} \ar@/^1em/[ll]^{\sH}
   }
 \end{equation}
 As $\kH$ is an over-ring of $\kA$, there is a morphism $\nu:(X,\kH)\to(X,\kA)$. In the case of commutative curves, it is just the
 normalization of $X$. This morphism induces the direct image functor $\nu_*:\kH\Md\to\kA\Md$ and its left and right adjoints,
 respectively, $\nu^*$ (the inverse image functor) and $\nu^!$. We show that they coincide (in some cases up to twist) with
 the compositions of the functors from the diagram \eqref{dia}.
 
 \begin{theorem}\label{k3} 
 \begin{enumerate}
 \item  $\nu_*\simeq\sG\tsF$ and $\nu^!\simeq\tsG\sH$.
 
 \item  $\tsG\sF\simeq\kC\*_\kH\nu^*$ and $\sG\tsH\simeq\nu_*(\kC'\*_\kH\_)$, where $\kC=\Hom_\kA(\kH,\kA)$ is the conductor
 of $\kH$ in $\kA$ and $\kC'=\Hom_\kH(\kC,\kH)$ is its dual $\kH$-module.
 \end{enumerate}
 \end{theorem}
 \begin{proof}
  We prove the equalities \emph{1}, the equalities \emph{2} are proved analogously.
  
  As $e_1\tkP=\kH$ as an $\kA\mt\kH$-bimodule, 
  \[
   \sG\tsF(\kM)=\Hom_{\tkA}(\kP,\tkP\*_\kH\kM)\simeq e_1\tkP\*_\kH\kM=\kM
   \]
  considered as an $\kA$-module, and it is just $\nu_*\kM$. On the other hand,
  \begin{align*}
  \tsG\sH(\kN)&=\Hom_{\tkA}(\tkP,\Hom_\kA(\kP\ch,\kN))\simeq e_{n+1}\Hom_\kA(\kP\ch,\kN)\\
  					&\simeq \Hom_\kA(\kP\ch e_{n+1},\kN)\simeq\Hom_\kA(\kH,\kN)=\nu^!\kN.
  \end{align*}
 \end{proof}
  
 \subsection{Rational case: Tilting}
 \label{sk2} 
 
 Now we suppose that the \ncc\ $(X,\kA)$ is \emph{rational over an algebraically closed field} $\aK$. 
 We keep the notations of the preceding subsection. According to Corollary~\ref{512}, the hereditary algebra $\kH$
 has a tilting module $\kT$ such that $(\End_\kH\kT)\op=\La$ is a Ringel canonical algebra. Set $\kT'=\tsF(\kT)$, $\kQ=\kQ_n$.
 
 \begin{theorem}\label{tilt} 
 \begin{enumerate}
 \item   $\tkT=\kQ[-1]\+\kT'$ is a \emph{tilting complex} for $\tkA$, i.e. $\tkT\in\cD\sm\tkA$, $\hom_{\cD\tkA}(\tkT,\tkT[m])=0$
 for $m\ne0$ and $\tkT$ generates $\cD\tkA$. Therefore, the functor $\rhom_{\cD\tkA}(\tkT,\_\,)$ establishes an equivalence
 $\cD\tkA\simeq\cD\tLa$, where $\tLa=(\End_{\cD\tkA}\tkT)\op$.
 
 \item  The algebra $\tLa$ is quasi-hereditary.
 
 \item  $\gdim\tLa\le 2n+1$.
 \end{enumerate}
 \end{theorem}
 \begin{proof}
  \textit{1.}\, As $\tkP$ is a projective right $\kH$-module, $\widetilde{\LF}$ is an exact full embedding. It also maps perfect complexes 
  to perfect complexes by Theorem~\ref{DFG}.\textit{3}. Therefore, $\kT'\in\cD\sm\tkA$, $\hom_{\cD\tkA}(\kT',\kT'[m])=0$ for $m\ne0$ 
  and $\kT'$ generates $\im\widetilde{\LF}$. On the other hand, $\kQ$ generates $\ker\tsG\simeq\kQ\Md$. Since 
  $\gnr{\ker\widetilde{DG},\im\widetilde{LF}}$ is a semi-orthogonal decomposition of $\cD\tkA$, $\tkT$ generates $\cD\tkA$
  and $\hom_{\cD\tkA}(\kT',\kQ[m])=0$ for all $m$.
  As $\dim\supp\kQ=0$, $\ext^k_{\tkA}(\kQ,\kM)=H^0(X,\Ext^k_{\tkA}(\kQ,\kM))$ for every quasi-coherent $\tkA$-module $\kM$. A locally
  projective resolution of $\kQ$ is $0\to\kI_{\tkP}\to\tkA\to\kQ\to0$, hence $\kQ\in\cD\sm\tkA$ and $\ext^k_{\tkA}(\kQ,\kM)=0$
  for $k>1$. Moreover, if $\kM$ is a $\kQ$-module, that is $\kI_{\tkP}\kM=0$, then $\Hom_{\tkA}(\kI_{\tkP},\kM)=0$, since
  $\kI_{\tkP}=\kI_{\tkP}^2$. Hence $\ext^1_{\tkA}(\kQ,\kM)=0$. Evidently, $\hom_{\tkA}(\kQ,\kT')=0$, whence 
  $\hom_{\cD\tkA}(\tkT,\tkT[m])=0$ for $m\ne0$, which accomplishes the proof.
  
  \smallskip
  \textit{2.}\, The algebra $\tLa$ can be considered as the algebra of triangular matrices
  \[
   \tLa=\mtr{Q&E\\0&\La},
  \]
  where $Q=H^0(X,\kQ),\,\La=\End_{\kH}\kT$ and $E=\ext^1_{\tkA}(\kQ,\kT')$. We have already seen that there is a heredity
  relating chain of length $n-1$ between $\kQ$ and $\kQ_1$. On the other hand, the algebra $\La$ is \emph{triangular}, i.e. contains
  a set of orthogonal idempotents $\lst fs$ such that $f_i\La f_i$ is semi-simple (in our case equals $\aK$), while $f_i\La f_j=0$
  if $j>i$. One easily sees that if an algebra $A$ can be presented as a matrix algebra of the form
  \[
   A=\mtr{A_1&B\\0&A_2},
  \] 
  where $A_2$ is semi-simple and $A_1$ is quasi-hereditary, then $A$ is also quasi-hereditary. Therefore, $\tLa$ is quasi-hereditary.
  
  \smallskip
  \textit{3.}\, One easily sees that $\gdim\La\le2$. On the other hand, $\gdim Q\le 2n-2$ by Corollary~\ref{bound}. Then the inequality
  $\gdim\tLa\le2n+1$ follows from \citep[ p.~407,~Corollary~4$'$]{pr}.
 \end{proof}
 
  Thus, every rational \ncc\ over an algebraically closed field has a categorical resolution by a finite dimensional quasi-hereditary 
  algebra.
  
  \medskip
   Recall that, for a triangulated category $\cT$, its \emph{Rouquier dimension} $\dim\cT$ is defined as the smallest $d$ such that
   $\gnr{T}_{d+1}=\cT$ for some object $T$ \citep{rou}. Here $\gnr{T}_1$ consists of direct summands of direct sums of shifts of $T$ 
   and $\gnr{T}_{k+1}$ consists of direct summands of the objects $A$ such that there is an exact triangle $B\to A\to C\to B[1]$,
   where $B\in\gnr{T}_k$ and $C\in\gnr{T}_1$.
  
  \begin{corollary}\label{ti2} 
   $\dim\cD\sm\kA\le2n+1$, where $\dim$ means the dimension of Rouquier of a triangulated category \citep{rou}. Namely, 
   $\gnr{\kG}_{2n+2}=\cD\sm\kA$, where $\kG=\kT\+\bop_{i=1}^n\kA_i/\kA_{n+1.i}[-1]$.
  \end{corollary}
  \begin{proof}
   Indeed, $\gnr{\tLa}_{2n+2}=\cD\sm\tLa$ by \cite[Proposition~7.4]{rou}. As the equivalence $\cD\tLa\simeq\cD\tkA$ maps
    $\tLa$ to $\tkT$, $\gnr{\tkT}_{2n+2}=\cD\sm\tkA$. Then $\gnr{\DG\tkT}_{2n+2}=\cD\sm\kA$. Note that $\sG\kM=e_1\kM$
   for any $\tkA$-module $\kM$. Therefore, $\sG\kT'=\kT$ and $\sG\kQ=\bop_{i=1}^n\kA_i/\kA_{n+1,i}$. It accomplishes the proof.
  \end{proof}
  
  If the curve $\kA$ is commutative, the hereditary curve $\kH$ is regular and the algebra $\La$ is hereditary (just a product of 
  Kronecker algebras).  In this case the estimate in Corollary~\ref{ti2} is $2n$ instead of $2n+1$. It generalizes the result of 
  \citep{bd}, where the curves of level $1$ were considered.

 \section*{Acknowledgements}
 
 The results of this paper were mainly obtained during the stay of the second author at the
Max-Plank-Institut f\"ur Mathematik. Its final version was prepared during the visit of the second and the third
author to the Institute of Mathematics of the K\"oln University.


\end{document}